\documentclass[12pt]{amsart}
\usepackage{fullpage}
\usepackage{amsmath}
\usepackage{graphics}
\usepackage[all]{xy}
\xyoption{v2} \xyoption{knot}

\let\url=\undefined
\usepackage
[bookmarks,colorlinks,linkcolor=blue,citecolor=blue,urlcolor=blue] {hyperref}

\title{Quotients of the Multiplihedron as Categorified Associahedra}
\author {Stefan Forcey}

\thanks{Thanks to {\Xy-pic} for the diagrams. }
\email{sforcey@tnstate.edu}
\address{Department of Physics and Mathematics\\
Tennessee State University\\
Nashville, TN 37209 \\
USA\\
}


\swapnumbers \theoremstyle{plain}
\newtheorem{theorem}{Theorem}[section]
\newtheorem{lemma}[theorem]{Lemma}

\theoremstyle{definition}
\newtheorem{definition}[theorem]{Definition}
\newtheorem{example}[theorem]{Example}

\theoremstyle{remark}
\newtheorem{remark}[theorem]{Remark}

\def\cal#1{\mathcal{#1}}
\newcommand{\bcal}[1]{\mbox{\boldmath${\cal {#1}}$}}

\begin{document}

\begin{abstract}
We describe a new sequence of polytopes which characterize $A_{\infty}$ maps from a
topological monoid to an $A_{\infty}$ space.
Therefore each of these polytopes is a quotient
of the corresponding multiplihedron.
Our sequence of polytopes is demonstrated not to be
 combinatorially equivalent to the associahedra, as was previously assumed in both topological and categorical literature.
  They are given the new collective
 name composihedra.
  We point out how these polytopes are used to parameterize compositions in the
formulation of the theories of
 enriched bicategories and pseudomonoids in a monoidal bicategory.
 We also present a simple algorithm for determining the extremal points in Euclidean space whose convex
hull  is the $n^{th}$ polytope in the sequence of composihedra, that is, the $n^{th}$ composihedron $\cal{CK}(n)$.
$$
\xy 0;/r.25pc/:
   (-4,3)*=0{}="a1";
  (4,3)*=0{}="a2";
  (6,-2)*=0{}="a3";
  (0,-7)*=0{}="a4";
  (-6,-2)*=0{}="a5";
  (-6,6.5)*=0{}="b1";
  (6,6.5)*=0{}="b2";
  (10,-4)*=0{}="b3";
  (0,-12)*=0{}="b4";
  (-10,-4)*=0{}="b5";
  (-8,10.5)*=0{}="c1";
  (8,10.5)*=0{}="c2";
  (13,7.5)*=0{}="c3";
  (16,0)*=0{}="c4";
  (14,-6)*=0{}="c5";
  (4,-14)*=0{}="c6";
  (0,-16)*=0{}="c7";
  (-4,-14)*=0{}="c8";
  (-14,-6)*=0{}="c9";
  (-16,0)*=0{}="c10";
  (-13,7.5)*=0{}="c11";
  (-10.5,-11.75)*=0{}="k1";
   (10.5,-11.75)*=0{}="k2";
    (0,12.5)*=0{}="k3";
     (0,0)*=0{}="k4";
   (-16,3.5)*=0{}="ck1";
    (16,3.5)*=0{}="ck2";
     (0,-4)*=0{}="ck3";
 (-9,-1.4)*{\hole}="x"; (-9,-1.4)*{\hole}="x";
(-6.8,4.4)*{\hole}="y";
  (9,-1.4)*{\hole}="w";
  (0,-12)*{\hole}="v";
 (6.8,4.4)*{\hole}="z";
 (-12,-12)*=0{}="ar1";
 (-19,-16)*=0{}="ar2";
(12,-12)*=0{}="ar3";
 (19,-16)*=0{}="ar4";
"ar1";"ar2" \ar@{->>}; "ar1";"ar2"\ar@{->>};
"ar3";"ar4" \ar@{-};
  "a1";"a2" \ar@{-}; "a1";"a2"\ar@{-};
  "a2";"a3"  \ar@{-};
    "a3";"a4" \ar@{-};
  "a4";"a5" \ar@{-};
  "a5";"a1" \ar@{-};
 "b1";"b2" \ar@{-};
 "b2";"b3" \ar@{-};
 "b3";"b4" \ar@{-};
 "b4";"b5" \ar@{-};
 "b5";"b1" \ar@{-};
 "c1";"c2"\ar@{-};
  "c2";"c3"  \ar@{-};
    "c3";"c4" \ar@{-};
  "c4";"c5" \ar@{-};
"c5";"c6"\ar@{-};
  "c6";"c7"  \ar@{-};
    "c7";"c8" \ar@{-};
  "c8";"c9" \ar@{-};
"c9";"c10"\ar@{-};
  "c10";"c11"  \ar@{-};
  "c11";"c1"  \ar@{-};
"b1";"c1"\ar@{-}; "b2";"c2"\ar@{-}; "b3";"c5"\ar@{-};
  "b4";"c6"  \ar@{-};
  "b4";"c8"  \ar@{-};
  "b5";"c9"  \ar@{-};
"a1";"y"  \ar@{-}; "y";"c11"  \ar@{-}; "a2";"z"  \ar@{-}; "z";"c3"  \ar@{-}; "a3";"w"  \ar@{-};
"w";"c4"  \ar@{-}; "a4";"v"  \ar@{-}; "v";"c7"  \ar@{-}; "a5";"x"  \ar@{-}; "x";"c10"  \ar@{-};
     \endxy
$$
\vspace{-.5in}
$$
\xy 0;/r.25pc/:
   (-4,3)*=0{}="a1";
  (4,3)*=0{}="a2";
  (6,-2)*=0{}="a3";
  (0,-7)*=0{}="a4";
  (-6,-2)*=0{}="a5";
  (-6,5.5)*=0{}="b1";
  (6,5.5)*=0{}="b2";
  (10,-4)*=0{}="b3";
  (0,-12)*=0{}="b4";
  (-10,-4)*=0{}="b5";
  (-8,10.5)*=0{}="c1";
  (8,10.5)*=0{}="c2";
  (13,7.5)*=0{}="c3";
  (16,0)*=0{}="c4";
  (14,-6)*=0{}="c5";
  (4,-14)*=0{}="c6";
  (0,-16)*=0{}="c7";
  (-4,-14)*=0{}="c8";
  (-14,-6)*=0{}="c9";
  (-16,0)*=0{}="c10";
  (-13,7.5)*=0{}="c11";
  (-13,-10)*=0{}="k1";
   (13,-10)*=0{}="k2";
    (0,12.5)*=0{}="k3";
     (0,-2)*=0{}="k4";
   (-14,3.25)*=0{}="ck1";
    (14,3.25)*=0{}="ck2";
     (0,-4)*=0{}="ck3";
     (-8.209,.251)*{\hole}="m";
  (8.209,.251)*{\hole}="n";
 (-9,-1.4)*{\hole}="x"; (-9,-1.4)*{\hole}="x";
(-6.8,4.4)*{\hole}="y";
  (9,-1.4)*{\hole}="w";
  (0,-12)*{\hole}="v";
 (6.8,4.4)*{\hole}="z";
 (-12,-12)*=0{}="ar1";
 (-19,-16)*=0{}="ar2";
(12,-12)*=0{}="ar3";
 (19,-16)*=0{}="ar4";
"ar3";"ar4"\ar@{->>};
"ar3";"ar4" \ar@{-};
 "b1";"b2" \ar@{-};"b1";"b2" \ar@{-};
 "b2";"b3" \ar@{-};
 "b3";"b4" \ar@{-};
 "b4";"b5" \ar@{-};
 "b5";"b1" \ar@{-};
 "c1";"c2"\ar@{-};
  "c2";"ck2"  \ar@{-};
  "ck2";"c5" \ar@{-};
"c5";"c6"\ar@{-};
  "c6";"c7"  \ar@{-};
    "c7";"c8" \ar@{-};
  "c8";"c9" \ar@{-};
"c9";"ck1"\ar@{-};
  "ck1";"c1"  \ar@{-};
"b1";"c1"\ar@{-}; "b2";"c2"\ar@{-}; "b3";"c5"\ar@{-};
  "b4";"c6"  \ar@{-};
  "b4";"c8"  \ar@{-};
  "b5";"c9"  \ar@{-};
"ck3";"v"  \ar@{-}; "v";"c7"  \ar@{-};
 "ck1";"m"  \ar@{-};
"m";"ck3"  \ar@{-};
  "ck2";"n"  \ar@{-};
"n";"ck3"  \ar@{-};
   "ck1";"c1"  \ar@{-};
     \endxy
~\hspace{2in}
\xy 0;/r.25pc/:
   (-4,3)*=0{}="a1";
  (4,3)*=0{}="a2";
  (6,-2)*=0{}="a3";
  (0,-7)*=0{}="a4";
  (-6,-2)*=0{}="a5";
  (-6,6.5)*=0{}="b1";
  (6,6.5)*=0{}="b2";
  (10,-4)*=0{}="b3";
  (0,-12)*=0{}="b4";
  (-10,-4)*=0{}="b5";
  (-8,10.5)*=0{}="c1";
  (8,10.5)*=0{}="c2";
  (13,7.5)*=0{}="c3";
  (16,0)*=0{}="c4";
  (14,-6)*=0{}="c5";
  (4,-14)*=0{}="c6";
  (0,-16)*=0{}="c7";
  (-4,-14)*=0{}="c8";
  (-14,-6)*=0{}="c9";
  (-16,0)*=0{}="c10";
  (-13,7.5)*=0{}="c11";
  (-13,-10)*=0{}="k1";
   (13,-10)*=0{}="k2";
    (0,12.5)*=0{}="k3";
     (0,-2)*=0{}="k4";
     (-14,3.25)*=0{}="ck1";
    (14,3.25)*=0{}="ck2";
     (0,-4)*=0{}="ck3";
(-3.451,-4.124)*{\hole}="p";
(3.451,-4.124)*{\hole}="q";
(0,3)*{\hole}="o";
 (-9,-1.4)*{\hole}="x"; (-9,-1.4)*{\hole}="x";
(-6.8,4.4)*{\hole}="y";
  (9,-1.4)*{\hole}="w";
  (0,-12)*{\hole}="v";
 (6.8,4.4)*{\hole}="z";
 (-12,-12)*=0{}="ar1";
 (-19,-16)*=0{}="ar2";
(12,-12)*=0{}="ar3";
 (19,-16)*=0{}="ar4";
"ar1";"ar2" \ar@{->>}; "ar1";"ar2"\ar@{-};
  "a1";"o" \ar@{-};"a1";"o" \ar@{-};
"o";"a2" \ar@{-};
  "a2";"a3"  \ar@{-};
    "a3";"q" \ar@{-};
"q";"a4" \ar@{-};
  "a4";"p" \ar@{-};
"p";"a5" \ar@{-};
  "a5";"a1" \ar@{-};
 "c11";"k3"\ar@{-};
  "k3";"c3"  \ar@{-};
    "c3";"c4" \ar@{-};
  "c4";"k2" \ar@{-};
  "k2";"c7"  \ar@{-};
    "c7";"k1" \ar@{-};
"k1";"c10"\ar@{-};
  "c10";"c11"  \ar@{-};
"a1";"c11"  \ar@{-};
"a2";"c3"  \ar@{-}; 
"a3";"c4"  \ar@{-}; 
 "a4";"c7"  \ar@{-};
  "a5";"c10"  \ar@{-}; 
"k1";"k4" \ar@{-};
"k2";"k4" \ar@{-};
"k3";"k4" \ar@{-};
     \endxy
$$
\vspace{-.5in}
$$
\xy 0;/r.25pc/:
   (-4,3)*=0{}="a1";
  (4,3)*=0{}="a2";
  (6,-2)*=0{}="a3";
  (0,-7)*=0{}="a4";
  (-6,-2)*=0{}="a5";
  (-6,6.5)*=0{}="b1";
  (6,6.5)*=0{}="b2";
  (10,-4)*=0{}="b3";
  (0,-12)*=0{}="b4";
  (-10,-4)*=0{}="b5";
  (-8,10.5)*=0{}="c1";
  (8,10.5)*=0{}="c2";
  (13,7.5)*=0{}="c3";
  (16,0)*=0{}="c4";
  (14,-6)*=0{}="c5";
  (4,-14)*=0{}="c6";
  (0,-17.25)*=0{}="c7";
  (-4,-14)*=0{}="c8";
  (-14,-6)*=0{}="c9";
  (-16,0)*=0{}="c10";
  (-13,7.5)*=0{}="c11";
  (-13.5,-10)*=0{}="k1";
   (13.5,-10)*=0{}="k2";
    (0,10)*=0{}="k3";
     (0,-2)*=0{}="k4";
     (-13.5,2)*=0{}="ck1";
    (13.5,2)*=0{}="ck2";
     (0,-5.25)*=0{}="ck3";
(-2.877,-3.705)*{\hole}="f";
(2.877,-3.705)*{\hole}="g";
 (-9,-1.4)*{\hole}="x"; (-9,-1.4)*{\hole}="x";
(-6.8,4.4)*{\hole}="y";
  (9,-1.4)*{\hole}="w";
  (0,-12)*{\hole}="v";
 (6.8,4.4)*{\hole}="z";
"ck1";"f" \ar@{-};"ck1";"f" \ar@{-};
"f";"ck3"  \ar@{-};
  "ck2";"g" \ar@{-};
"g";"ck3"  \ar@{-};
   "k1";"k4" \ar@{-};
"k2";"k4" \ar@{-};
"k3";"k4" \ar@{-};
"c7";"ck3" \ar@{-};
"ck1";"k3" \ar@{-};
"k3";"ck2" \ar@{-};
"ck2";"k2" \ar@{-};
"k2";"c7" \ar@{-};
"c7";"k1" \ar@{-};
"k1";"ck1" \ar@{-};
     \endxy
$$
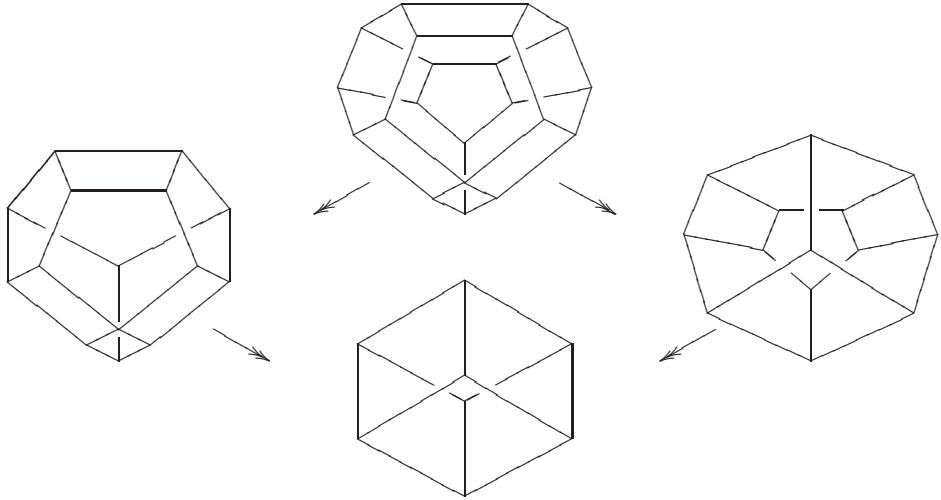
Figure 1: The cast of characters. Left to right:  $\cal{CK}(4)$, $\cal{J}(4)$, 3-d cube, and $\cal{K}(5)$.
\end{abstract}

\maketitle \tableofcontents
\section{Introduction}
Categorification, as described in \cite{Baez1}, refers to the process of creating a new mathematical
theory by 1) choosing a demonstrably useful concept that you understand fairly well,  2)
replacing some of the sets in its definition with categories, and 3) replacing some of the interesting
equalities with morphisms.
The associahedra are a sequence of polytopes, invented by Stasheff in \cite{Sta},
 whose face poset is defined to correspond to
bracketings of a given list of elements from a set. If instead we take lists of elements from
a category, and define a sequence of polytopes whose faces correspond to either bracketings or
 to certain maps in that
category, then we may naively describe our new definition as a categorified version of the associahedra.
Of course there may be many interesting ways to make this idea precise. Here we focus on just one, which arises
naturally in the study of both topological monoids and category theory.

Categorification of the concept of category itself can be achieved by replacing sets of
morphisms (and/or sets of objects) by categories.
This implies considering enriched (or
internal) categories of \textbf{Cat}, the category of categories.
There is room for
 the composition laws in
enriched or internal categories in  \textbf{Cat} to be weakened.
 To achieve this in a coherent way the
familiar
definitions of bi- and tricategory
utilize the
Stasheff polytopes, or associahedra, as the underlying shapes of axiomatic commuting diagrams.

A parallel categorification of enriched categories creates a definition
    allowing the hom-objects to
  be located in a monoidal bicategory $\bcal{W}.$  This is known as the theory of enriched bicategories.
Rather than the associahedra,  it is a new sequence of polytopes which arises in the
corresponding coherence axioms of enriched bicategories.
  This new sequence
comprises the categorified version of the associahedra which we will be studying. Since they appear in the
composition laws of enriched categories we have chosen to refer to them collectively as the composihedra, or singulary as
the $n^{th}$ composihedron, denoted $\cal{CK}(n)$. The capital ``$\cal{C}$'' stands for ``compose'', or for ``categorify,''
or for ``cone''--
the last since the $n^{th}$ composihedron can be seen as being a subdivision of the topological
 cone of the $n^{th}$ associahedron.
Indeed the polytope $\cal{CK}(n)$ is of dimension $n-1$, while the polytope $\cal{K}(n)$ is of dimension $n-2.$
Decompositions of the boundaries of the earliest terms in our new sequence of polytopes have been seen before,
in the classic sources on enriched
categories such as \cite{Kelly}, the definition of enriched bicategories in \cite{Sean}, and in the
single-object version of the latter: pseudomonoids as defined in \cite{paddy}.

The other half of our title refers to the fact that the polytopes we study here also arise in the
study of $A_n$ and $A_{\infty}$ maps. The multiplihedra were invented by Stasheff,
described by Iwase and Mimura, and generalized by Boardman and Vogt.
They represent the fundamental structure of a weak map between weak structures,
 such as weak $n$-categories or $A_n$ spaces. They form a bimodule over the associahedra,
  and collapse under a quotient map to become the associahedra in the special case of a strictly associative
  monoid as range.
 In the case of a strictly associative monoid as domain
   the multiplihedra collapse to form our new family of polytopes. This is pictured above in Figure 1, in dimension 3.
   The multiplihedron is at the top, composihedron at left, associahedron on the
   right, and the cube which results in the case of both strict domain and range structures at the bottom.

In section~\ref{two} we  briefly review the appearance of polytope sequences in topology and category theory.
In section~\ref{three} we  provide a complete recursive definition of our new polytopes, as well as a description of them
as quotients of the multiplihedra. In section~\ref{four} we go over some basic combinatorial
 results about the composihedra.
In section~\ref{five} we present an alternative definition of the composihedra as a convex hull, based upon the convex hull
realization of the multiplihedra in \cite{multi} and which reflects the quotienting process. In section~\ref{proofsec} we
prove that the convex hull defined in section~\ref{five} is indeed combinatorially equivalent to the complex defined in
section \ref{three}.

A word of introduction is appropriate in regard to the
convex hull algorithm in sections~\ref{five} and~\ref{proofsec}. In the paper on the multiplihedra \cite{multi}
we described how to represent Boardman and Vogt's spaces of painted
trees with $n$ leaves as convex polytopes which are combinatorially equivalent to the CW-complexes
 described by Iwase and Mimura.
 Our algorithm for the vertices of the polytopes is flexible
in that it allows an initial choice of a constant $q$ between zero and one. In the limit as $q\to
1$ the convex hull approaches that of Loday's convex hull representation of the associahedra as
described in \cite{loday}. The limit as $q\to 1$ corresponds to the case for which the mapping
is a homomorphism in that it
strictly respects the
 multiplication. The
limit as $q\to 0$ represents the case for which multiplication in the domain of the morphism in
question is strictly associative. In the limit as $q\to 0$ the convex hulls of the
 multiplihedra approach our newly discovered sequence of
polytopes, the composihedra.

There are two projects for the future that are supported by this work. One is to make rigorous the
implication that enriched bicategories may be exemplified by certain maps of topological monoids.
It could be hoped that if this endeavor
 is successful
that $A_{\infty}$ categories and their maps might also be amenable to the same approach,
yielding more interesting examples. The other project already underway is to
extend the
concept of quotient multiplihedra described here to the graph associahedra
 introduced by Carr and Devadoss, in
\cite{dev}.

There is also a philosophical conclusion to be argued from the results of this work.
 Historically, weak $2$ and
$3$-categories were defined using  the associahedra, which form an operad of topological spaces.
Then the operad structure was taken as fundamental
in many of the functioning definitions of weak $n$-category, as described in \cite{lst3} and \cite{leinster1}.
 Again historically, weak maps of bi- and tri-categories were defined using the multiplihedra,
which form a 2-sided operad bimodule over the associahedra. More recently the operad bimodule
structure has been used to define weak maps of
weak $n$-categories, in \cite{Hess} and \cite{bat}. Thus the facts that the composihedra are
 used for defining  enriched categories and bicategories, and
  that they form an operad bimodule (left-module over the associahedra and right-module over the associative operad
  $t(n) = {*}$ ) lead us to propose that enriching over a weak $n$-category should in general be accomplished by
  use of operad bimodules as well.
%
The philosophy here is that the structure of a bimodule will take into account the
weakness of
the base of enrichment, (where a  weakly associative product is used to form the domain for composition)
  as well as providing for the weakness of the enriched composition
itself.

\section{Polytopes in topology and categories}\label{two}

Here we review the appearance of fundamental families of polytopes in the axioms of higher dimensional category theory.
In both topology and in category theory, the use of these polytopes has proven to be a source of important clues rather
than the final solution. The algebraic structure of the polytope sequence is more important than its combinatorial structure,
although certainly one depends on the other. Thus the operad structure on
the associahedra can be seen as foreshadowing the
intuition for the use of actions of the operad of little $n$-cubes to recognize loop spaces,
 as well as the use of $n$-operad actions in Batanin's definition of $n$-category
\cite{bat}.

\subsection{Associahedron}

The associahedra are the famous sequence of polytopes denoted $\cal{K}(n)$ from \cite{Sta}
which characterize the structure of weakly associative products. $\cal{K}(1) = \cal{K}(2) = $ a single point,
$\cal{K}(3)$ is the line segment, $\cal{K}(4)$ is the pentagon, and $\cal{K}(5)$ is the following 3d shape:
$${\cal K}(5) = 
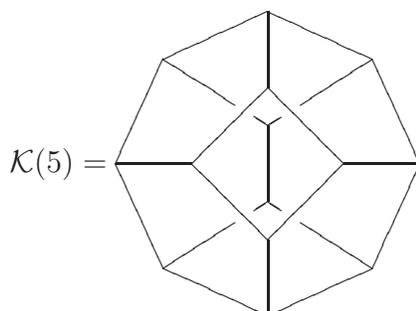
\xy 0;/r.3pc/:
   (0,16)*=0{}="a";
  (11,11)*=0{}="b";
  (16,0)*=0{}="c";
  (11,-11)*=0{}="d";
  (0,-16)*=0{}="e";
  (-11,-11)*=0{}="f";
  (-16,0)*=0{}="g";
  (-11,11)*=0{}="h";
  (0,8)*=0{}="2a";
  (8,0)*=0{}="2b";
  (0,-8)*=0{}="2c";
  (-8,0)*=0{}="2d";
(0,4)*=0{}="3a";
  (0,-4)*=0{}="3b";
 (-2.44,5.56)*{\hole}="p";
 (2.44,5.56)*{\hole}="q";
 (-2.44,-5.56)*{\hole}="r";
 (2.44,-5.56)*{\hole}="s";
  "a";"b" \ar@{-}; "a";"b"\ar@{-};
  "b";"c"  \ar@{-};
    "c";"d" \ar@{-};
  "d";"e" \ar@{-};
 "e";"f" \ar@{-};
 "f";"g" \ar@{-};
 "g";"h" \ar@{-};
 "h";"a" \ar@{-};
 "2a";"2b"\ar@{-};
  "2b";"2c"  \ar@{-};
    "2c";"2d" \ar@{-};
  "2d";"2a" \ar@{-};
"2a";"a"\ar@{-};
  "2b";"c"  \ar@{-};
    "2c";"e" \ar@{-};
  "2d";"g" \ar@{-};
"3a";"p"\ar@{-};
"p";"h"\ar@{-};
  "3b";"r"  \ar@{-};
  "r";"f"  \ar@{-};
"3a";"q"\ar@{-};
"q";"b"\ar@{-};
"3b";"s"\ar@{-};
  "s";"d"  \ar@{-};
  "3b";"3a"  \ar@{-};
%
%
     \endxy
$$
The original examples of weakly associative product
 structure are the $A_n$ spaces, topological $H$-spaces with weakly associative
 multiplication of points. Here ``weak'' should be understood as ``up to homotopy.'' That is,
 there is a path in the space
 from $(ab)c$ to $a(bc).$ An $A_{\infty}$-space $X$ is characterized by its admission of an action
 $$\cal{K}(n) \times X^n \to X$$
 for all $n.$

 Categorical examples begin with the monoidal categories as defined in \cite{MacLane},
 where there is a
 weakly associative tensor product of objects. Here ``weak'' officially means ``naturally isomorphic.''
 There is a natural
 isomorphism $\alpha: (U\otimes V) \otimes W \to U\otimes (V \otimes W).$

Recall that a {\it monoidal category} is a category ${\cal V}$
      together with a functor
      $\otimes: {\cal V}\times{\cal V}\to{\cal V}$ such that
     $\otimes$ is  associative up to the coherent natural isomorphism $\alpha$. The coherence
      axiom is given by a commuting pentagon, which is a copy of $\cal{K}(4).$
      \noindent
        \begin{center}
        \resizebox{5.5in}{!}{
              $$
              \xymatrix@C=-25pt{
              &((U\otimes V)\otimes W)\otimes X \text{ }\text{ }
              \ar[rr]^{ \alpha_{UVW}\otimes 1_{X}}
              \ar[ddl]^{ \alpha_{(U\otimes V)WX}}
              &&\text{ }\text{ }(U\otimes (V\otimes W))\otimes X
              \ar[ddr]^{ \alpha_{U(V\otimes W)X}}&\\\\
              (U\otimes V)\otimes (W\otimes X)
              \ar[ddrr]|{ \alpha_{UV(W\otimes X)}}
              &&&&U\otimes ((V\otimes W)\otimes X)
              \ar[ddll]|{ 1_{U}\otimes \alpha_{VWX}}
              \\\\&&U\otimes (V\otimes (W\otimes X))&&&
             }
             $$
             }
                        \end{center}

  Monoidal categories can be
 viewed as single object bicategories. In a bicategory each $\hom(a,b)$ is located
 in \textbf{Cat} rather than \emph{Set}. The composition of morphisms is not exactly associative: there is
 a 2-cell $\sigma: (f \circ g) \circ h \to f\circ (g\circ h ).$
 This associator obeys the same pentagonal commuting diagram as for monoidal categories, as seen in \cite{lst1}.

  Another iteration of categorification results in the
 theory of tricategories. These obey an axiom in which a commuting pasting diagram has
 the underlying  form of the
 associahedron $\cal{K}(5),$ as noticed by the authors of \cite{GPS}.
 The term ``cocycle condition''  for this axiom was popularized by Ross Street, and its connections to
homology are described in many of his papers, including \cite{StAlg}.
  The pattern continues in Trimble's definition
 of tetracategories, where $\cal{K}(6)$ is found as the underlying structure of the corresponding cocycle condition.
   The associahedra are also seen as the classifying spaces of certain
  categories of trees, as illustrated in \cite{lst1}, and as the foundation for the free strict $\omega$-categories
  defined in \cite{StAlg}.

  The associahedra are also the starting point for defining the
one-dimensional analogs of the full $n$-categorical comparison of
delooping and enrichment. One dimensional weakened versions of
enriched categories have been well-studied in the field of
differential graded algebras and $A_{\infty}$-categories, the many
object generalizations of Stasheff's $A_{\infty}$-algebras
\cite{Sta}.
An $A_{\infty}$-category category is basically a category
 ``weakly''
 enriched over chain complexes of modules, where the weakening in
 this case is accomplished by summing the composition chain maps
 to zero
 (rather than by requiring commuting diagrams). It is also easily
 described as an algebra over a certain operad.

\subsection{Multiplihedron}

The complexes now known as the multiplihedra $\cal{J}(n)$ were first pictured by Stasheff, for $n \le 4$ in
\cite{sta2}. They were introduced in order to approach a full description of the category of
$A_{\infty}$ spaces by providing the underlying structure for morphisms which preserved the
structure of the domain space ``up to homotopy'' in the range. Recall that an $A_{\infty}$ space
itself is a monoid only ``up to homotopy,'' and is recognized by a continuous action of the
associahedra as described in \cite{Sta}. Thus the multiplihedra are used to characterize the
$A_{\infty}$-maps. A map $f:X\to Y$ between $A_{\infty}$-spaces is an $A_{\infty}$-map if there exists
an action
$$\cal{J}(n) \times X^n \to Y $$
for all $n$, which is equal to the action of $f$ for $n=1,$ and obeying associativity
constraints as described in \cite{sta2}.
 Stasheff described how to construct the 1-skeleton of these complexes in \cite{sta2}, but
stopped short of a full combinatorial description.
 Iwase and Mimura in \cite{IM}
give the first detailed definition of the sequence of complexes $\cal{J}(n)$ now known as the
multiplihedra, and describe their combinatorial properties.

Spaces of painted trees were first introduced by
Boardman and Vogt in \cite{BV1} to help describe multiplication in (and morphisms of) topological monoids that
are not strictly associative (and whose morphisms do not strictly respect that multiplication.)
 The $n^{th}$
multiplihedron is a  $CW$-complex whose vertices correspond to the unambiguous ways of multiplying
and applying an $A_{\infty}$-map to $n$ ordered elements of an $A_{\infty}$-space. Thus the
vertices correspond to the binary painted trees with $n$ leaves. In \cite{multi} a new realization of the multiplihedra
based upon a map from these binary painted trees to Euclidean space is used to
unite the approach to $A_n$-maps of Stasheff, Iwase and Mimura to that of Boardman and Vogt.

Here are the first few low dimensional multiplihedra. The  vertices are labeled, all but some of
those in the last picture. There the bold vertex in the large pentagonal facet has label
$((f(a)f(b))f(c))f(d)$ and the bold vertex in the small pentagonal facet has label $f(((ab)c)d).$
The others can be easily determined based on the fact that those two pentagons are copies of the
associahedron $\cal{K}(4),$ that is to say all their edges are associations.
$$
{\cal J}(1) = \bullet~{_{f(a)}}
  %
 %
$$
\\\\
$$
{\cal J}(2)= 
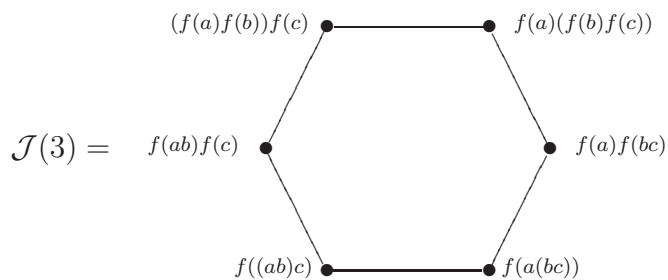
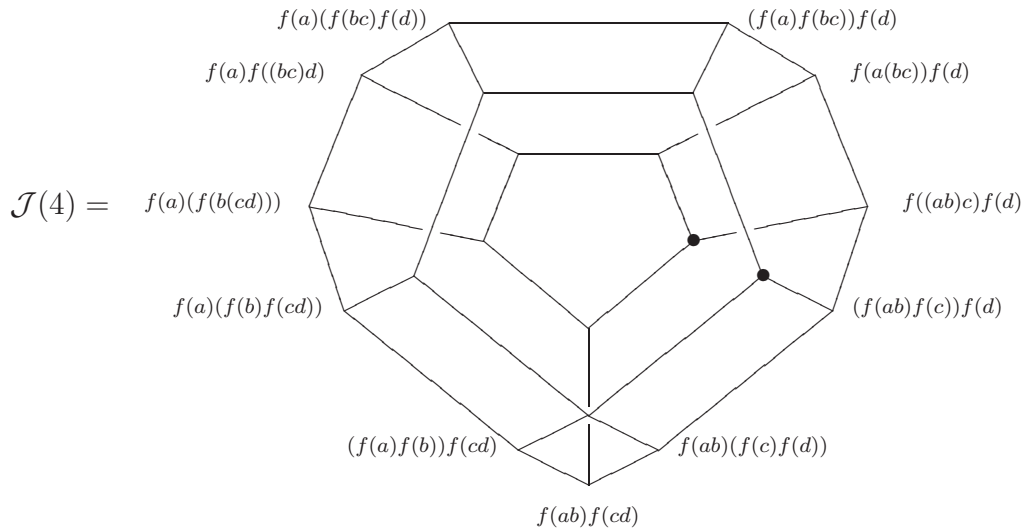
\xy 0;/r.25pc/:
    (-8,0)*{_{f(a)f(b)}~}="e";
  (8,0)*{~_{f(ab)}}="v6" \ar@{};
  "e"; "v6" \ar@{*{\bullet}-*{\bullet}} ;"e"; "v6" \ar@{-} \ar@{};
  \endxy
$$
\\\\
$$
{\cal J}(3)=\hspace{.75in} \xy 0;/r.08pc/:
 (-32,48);(0,48) *=0{^{(f(a)f(b))f(c)}\hspace{1.75in}}  \ar@{*{\bullet}-}; (-32,48);(0,48) *=0{} \ar@{-};
 (0,48);(32,48)*=0{^{f(a)(f(b)f(c))}\hspace{-1in}}  \ar@{*{\bullet}-};
  (32,48);(44,24)*=0{}  \ar@{-};
  (44,24); (56,0)*=0{^{f(a)f(bc)}\hspace{-.75in}}  \ar@{*{\bullet}-};
   (56,0); (44,-24)*=0{}  \ar@{-};
   (44,-24); (32,-48)*=0{^{f(a(bc))}\hspace{-.55in}}  \ar@{*{\bullet}-};
    (32,-48);(0,-48)*=0{}  \ar@{-};
    (0,-48); (-32,-48)*=0{^{f((ab)c)}\hspace{.55in}}  \ar@{*{\bullet}-};
     (-32,-48);(-44,-24)*=0{}  \ar@{-};
     (-44,-24); (-56,0)*=0{^{f(ab)f(c)}\hspace{.75in}}  \ar@{*{\bullet}-};
      (-56,0);(-44,24)*=0{}  \ar@{-};
      (-44,24); (-32,48)*=0{}  \ar@{};
      \endxy
$$
\\\\
$$
{\cal J}(4) = \hspace{1in}\xy 0;/r.55pc/:
   (-4,3)*=0{}="a1";
  (4,3)*=0{}="a2";
  (6,-2)*=0{\bullet}="a3";
  (0,-7)*=0{}="a4";
  (-6,-2)*=0{}="a5";
  (-6,6.5)*=0{}="b1";
  (6,6.5)*=0{}="b2";
  (10,-4)*=0{\bullet}="b3";
  (0,-12)*=0{}="b4";
  (-10,-4)*=0{}="b5";
  (-8,10.5)*=0{^{f(a)(f(bc)f(d))}\hspace{1in}}="c1";
  (8,10.5)*=0{^{(f(a)f(bc))f(d)}\hspace{-1in}}="c2";
  (13,7.5)*=0{^{f(a(bc))f(d)}\hspace{-1in}}="c3";
  (16,0)*=0{^{f((ab)c)f(d)}\hspace{-1in}}="c4";
  (14,-6)*=0{^{(f(ab)f(c))f(d)}\hspace{-1in}}="c5";
  (4,-14)*=0{^{f(ab)(f(c)f(d))}\hspace{-1in}}="c6";
  (0,-16)*=0{}="c7";
  (0,-18)*=0{^{f(ab)f(cd)}};
  (-4,-14)*=0{^{(f(a)f(b))f(cd)}\hspace{1in}}="c8";
  (-14,-6)*=0{^{f(a)(f(b)f(cd))}\hspace{1in}}="c9";
  (-16,0)*=0{^{f(a)(f(b(cd)))}\hspace{1in}}="c10";
  (-13,7.5)*=0{^{f(a)f((bc)d)}\hspace{1in}}="c11";
 (-9,-1.4)*{\hole}="x"; (-9,-1.4)*{\hole}="x";
(-6.8,4.4)*{\hole}="y";
  (9,-1.4)*{\hole}="w";
  (0,-12)*{\hole}="v";
 (6.8,4.4)*{\hole}="z";
  "a1";"a2" \ar@{-}; "a1";"a2"\ar@{-};
  "a2";"a3"  \ar@{-};
    "a3";"a4" \ar@{-};
  "a4";"a5" \ar@{-};
  "a5";"a1" \ar@{-};
 "b1";"b2" \ar@{-};
 "b2";"b3" \ar@{-};
 "b3";"b4" \ar@{-};
 "b4";"b5" \ar@{-};
 "b5";"b1" \ar@{-};
 "c1";"c2"\ar@{-};
  "c2";"c3"  \ar@{-};
    "c3";"c4" \ar@{-};
  "c4";"c5" \ar@{-};
"c5";"c6"\ar@{-};
  "c6";"c7"  \ar@{-};
    "c7";"c8" \ar@{-};
  "c8";"c9" \ar@{-};
"c9";"c10"\ar@{-};
  "c10";"c11"  \ar@{-};
  "c11";"c1"  \ar@{-};
"b1";"c1"\ar@{-}; "b2";"c2"\ar@{-}; "b3";"c5"\ar@{-};
  "b4";"c6"  \ar@{-};
  "b4";"c8"  \ar@{-};
  "b5";"c9"  \ar@{-};
"a1";"y"  \ar@{-}; "y";"c11"  \ar@{-}; "a2";"z"  \ar@{-}; "z";"c3"  \ar@{-}; "a3";"w"  \ar@{-};
"w";"c4"  \ar@{-}; "a4";"v"  \ar@{-}; "v";"c7"  \ar@{-}; "a5";"x"  \ar@{-}; "x";"c10"  \ar@{-};
     \endxy
     ~
$$
\\\\

The multiplihedra  also appear in higher category theory. The definitions of  bicategory and
tricategory homomorphisms each include commuting pasting diagrams as seen in \cite{lst1} and
\cite{GPS} respectively. The two halves of the axiom for a bicategory homomorphism together form
the boundary of the multiplihedra $\cal{J}(3),$  and the two halves of the axiom for a tricategory
homomorphism together form the boundary of $\cal{J}(4).$ Since weak $n$-categories can be
understood as being the algebras of higher operads, these facts can be seen as the motivation for
defining morphisms of operad (and $n$-operad) algebras in terms of their bimodules. This definition
is mentioned in \cite{bat} and developed in detail in \cite{Hess}. In the latter paper it is
pointed out that the bimodules in question must be co-rings, which have a co-multiplication with
respect to the bimodule product over the operad.

The multiplihedra have also
appeared in several areas related to deformation theory and $A_{\infty}$ category theory. A
diagonal map is constructed for these polytopes in \cite{umble}. This allows a functorial monoidal
structure for certain categories of $A_{\infty}$-algebras and $A_{\infty}$-categories. A different,
possibly equivalent, version of the diagonal is presented in \cite{MS}. The 3 dimensional version
of the multiplihedron is called by the name Chinese lantern diagram in \cite{Yet}, and used to
describe deformation of functors. There is a forthcoming paper by Woodward and Mau in which a new
realization of the multiplihedra as moduli spaces of disks with additional structure is presented
\cite{Mau}. This realization  allows the authors  to define
$A_n$-functors as well as morphisms of cohomological field theories.

\subsection{Quotients of the multiplihedron}

The special multiplihedra in the case for which multiplication in the range is strictly
associative were found by
 Stasheff in \cite{sta2} to be precisely the associahedra.
  Specifically, the quotient of $\cal{J}(n)$ under
 the equivalence  generated by $(f(a)f(b))f(c)= f(a)(f(b)f(c))$ is combinatorially equivalent to
$\cal{K}(n+1).$ This projection is pictured on the right hand side of Figure 1, in dimension 3.
Recall that the edges of the multiplihedra
correspond to either an association $(ab)c\to a(bc)$ or to a preservation $f(a)f(b)\to f(ab).$
The
associations can either be in the range: $(f(a)f(b))f(c)\to f(a)(f(b)f(c))$; or the image of a
domain association:  $f((ab)c)\to f(a(bc)).$

  It was long assumed that the case for
which the
 domain was associative would likewise yield the associahedra, but we will demonstrate otherwise.
 The $n^{th}$ composihedron may be described as the quotient of the  $n^{th}$ multiplihedron under the
 equivalence generated by $f((ab)c)= f(a(bc))$. Of course this is implied by associativity in the domain,
 where $(ab)c = a(bc).$ We will take this latter view throughout, but we note that there may be interesting functions
 for which $f((ab)c)= f(a(bc))$ even if the domain is not strictly associative.

Here are the first few composihedra with vertices labeled
 as in the multiplihedra, but with the assumption that the domain is associative.
  (For these pictures the label actually appears over the vertex.)
  Notice how the $n$-dimensional  composihedron is a subdivided topological cone on the
 $(n-1)$-dimensional  associahedron. The Schlegel diagram is shown for $\cal{CK}(4),$ viewed with a copy of
 $\cal{K}(4)$ as the perimeter.
 \begin{tiny}
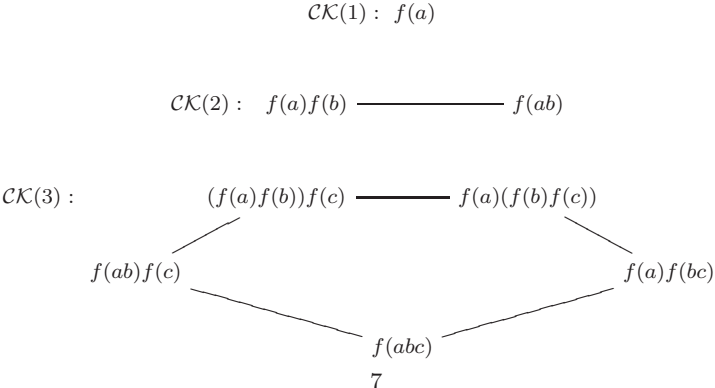
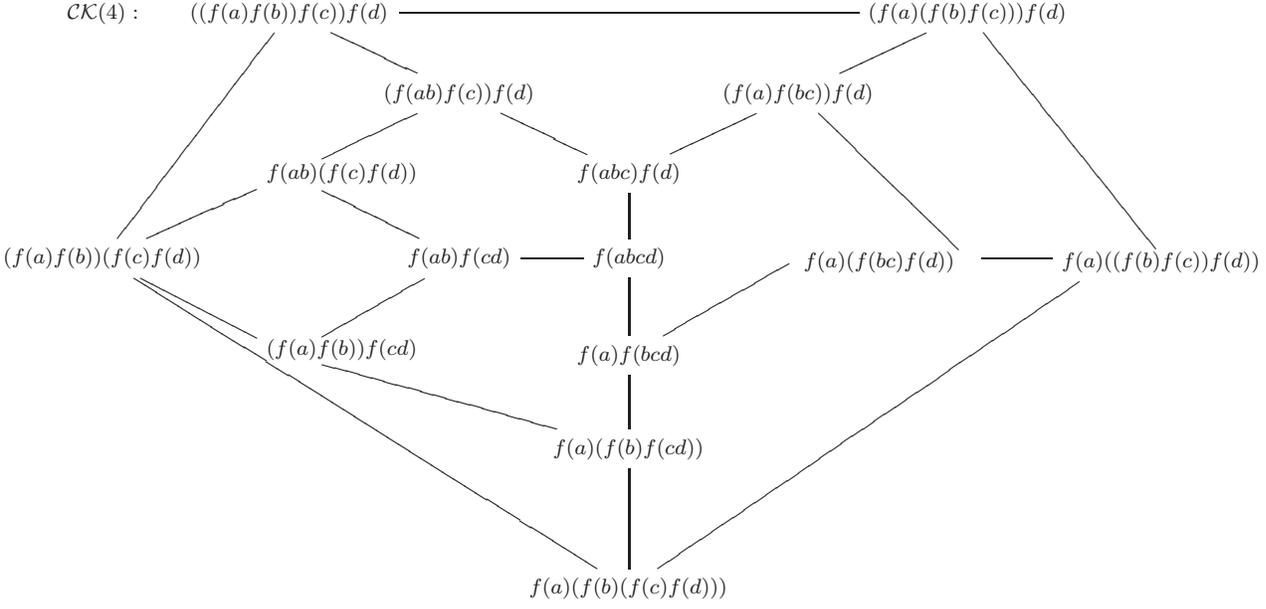
 $$\cal{CK}(1):~~~f(a)~~~~~~
 $$
 $$
 \cal{CK}(2):~~\xymatrix{ f(a)f(b)\ar@{-}[rr]&&f(ab)}~~~~~~~
 $$
 $$
 \cal{CK}(3): \xymatrix@R=3.25pt@C=2.25pt{
    &(f(a)f(b))f(c)
    \ar@{-}[rr]
    \ar@{-}[ddl]
    && f(a)(f(b)f(c))
    \ar@{-}[ddr]&\\\\
    f(ab)f(c)
    \ar@{-}[ddrr]
    &&&& f(a)f(bc)
      \ar@{-}[ddll]
      \\\\&& f(abc) &&&
 }
 $$
 $$
 \xymatrix@R=15.5pt@C=-9.5pt{\cal{CK}(4):&&((f(a)f(b))f(c))f(d)\ar@{-}[rrrr]\ar@{-}[rd]\ar@{-}[llddd]&&&&(f(a)(f(b)f(c)))f(d)\ar@{-}[rrddd]\ar@{-}[ld]
       \\
       &&&(f(ab)f(c))f(d)\ar@{-}[rd]\ar@{-}[ld]&&(f(a)f(bc))f(d)\ar@{-}[ld]\ar@{-}[ddr]
       \\
       && f(ab)(f(c)f(d))\hspace{-.55in}\ar@{-}[rd]&& f(abc)f(d)\ar@{-}[d]
       \\
       (f(a)f(b))(f(c)f(d))\ar@{-}[rru]\ar@{-}[rrd]\ar@{-}[rrrrdddd]&&& f(ab)f(cd)\ar@{-}[r]& f(abcd) &\txt{\\$f(a)(f(bc)f(d))$\\}\hspace{-.85in}\ar@{-}[ld]&~ &\txt{\\$f(a)((f(b)f(c))f(d))$\\}\hspace{-.5in}\ar@{-}[llldddd]\ar@{-}[l]&~
       \\
       &&\txt{\\$(f(a)f(b))f(cd)$\\\\}\hspace{-.55in}\ar@{-}[ru]\ar@{-}[rrd]&& f(a)f(bcd)\ar@{-}[u]
       \\
       &&&& f(a)(f(b)f(cd))\ar@{-}[u]
       \\\\
       &&&& f(a)(f(b)(f(c)f(d))) \ar@{-}[uu]
   }
 $$
\end{tiny}


   Any confusion can probably be traced to the fact that
  the two sequences of polytopes are identical for the first few terms. They
    diverge first in dimension three, at which dimension
  the associahedron has 9 facets and 14 vertices, while the composihedron has 10 facets and 15 vertices.
   Another similarity at this dimension is that
  both polytopes have exactly 6 pentagonal facets; the difference is in the number of quadrilateral facets.
  The difference is also clear from
   the fact that the associahedra are simple polytopes, whereas starting at dimension three the composihedra are not.

   Here are easily compared pictures of the associahedron and composihedron in dimension 3.
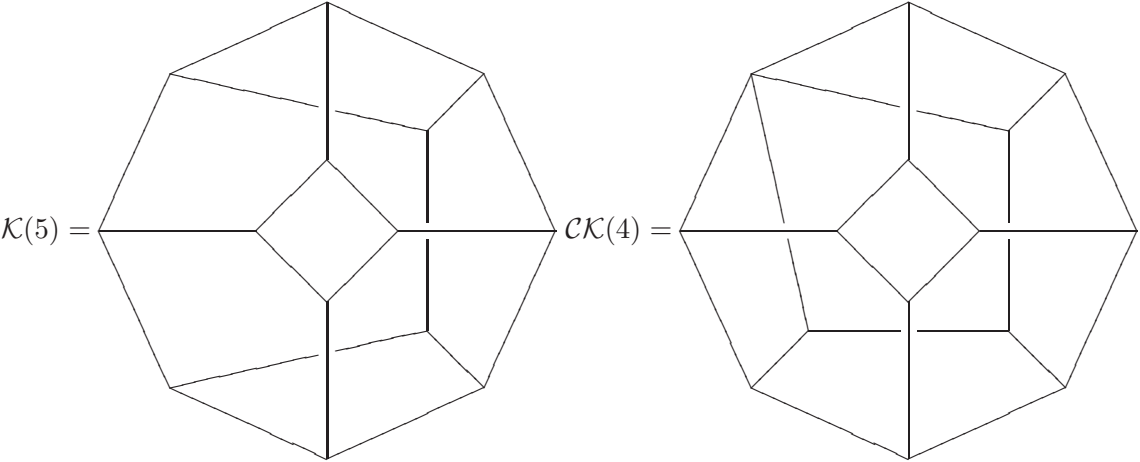
\begin{small}
$$
{\cal K}(5) = \xy 0;/r.45pc/:
   (0,16)*=0{}="a";
  (-11,11)*=0{}="b";
  (-16,0)*=0{}="c";
  (-11,-11)*=0{}="d";
  (0,-16)*=0{}="e";
  (11,-11)*=0{}="f";
  (16,0)*=0{}="g";
  (11,11)*=0{}="h";
  (0,5)*=0{}="2a";
  (-5,0)*=0{}="2b";
  (0,-5)*=0{}="2c";
  (5,0)*=0{}="2d";
(7,7)*=0{}="3a";
  (7,-7)*=0{}="3b";
  (-7,-7)*=0{}="3c";
 (0,8.556)*{\hole}="x"; (0,8.556)*{\hole}="x";
(7,0)*{\hole}="y";
  (0,-8.556)*{\hole}="w";
  (-8.556,0)*{\hole}="v";
 (0,-7)*{\hole}="z";
  "a";"b" \ar@{-}; "a";"b"\ar@{-};
  "b";"c"  \ar@{-};
    "c";"d" \ar@{-};
  "d";"e" \ar@{-};
 "e";"f" \ar@{-};
 "f";"g" \ar@{-};
 "g";"h" \ar@{-};
 "h";"a" \ar@{-};
 "2a";"2b"\ar@{-};
  "2b";"2c"  \ar@{-};
    "2c";"2d" \ar@{-};
  "2d";"2a" \ar@{-};
"2a";"a"\ar@{-};
  "2b";"c"  \ar@{-};
    "2c";"e" \ar@{-};
  "2d";"g" \ar@{-};
"3a";"h"\ar@{-};
  "3b";"f"  \ar@{-};
"3a";"y"\ar@{-};
"y";"3b"\ar@{-};
"3a";"x"\ar@{-};
  "x";"b"  \ar@{-};
"3b";"w"\ar@{-};
  "w";"d"  \ar@{-};
%
     \endxy
     ~
\cal{CK}(4) = \xy 0;/r.45pc/:
   (0,16)*=0{}="a";
  (-11,11)*=0{}="b";
  (-16,0)*=0{}="c";
  (-11,-11)*=0{}="d";
  (0,-16)*=0{}="e";
  (11,-11)*=0{}="f";
  (16,0)*=0{}="g";
  (11,11)*=0{}="h";
  (0,5)*=0{}="2a";
  (-5,0)*=0{}="2b";
  (0,-5)*=0{}="2c";
  (5,0)*=0{}="2d";
(7,7)*=0{}="3a";
  (7,-7)*=0{}="3b";
  (-7,-7)*=0{}="3c";
 (0,8.556)*{\hole}="x"; (0,8.556)*{\hole}="x";
(7,0)*{\hole}="y";
  (0,-8.556)*{\hole}="w";
  (-8.556,0)*{\hole}="v";
 (0,-7)*{\hole}="z";
  "a";"b" \ar@{-}; "a";"b"\ar@{-};
  "b";"c"  \ar@{-};
    "c";"d" \ar@{-};
  "d";"e" \ar@{-};
 "e";"f" \ar@{-};
 "f";"g" \ar@{-};
 "g";"h" \ar@{-};
 "h";"a" \ar@{-};
 "2a";"2b"\ar@{-};
  "2b";"2c"  \ar@{-};
    "2c";"2d" \ar@{-};
  "2d";"2a" \ar@{-};
"2a";"a"\ar@{-};
  "2b";"c"  \ar@{-};
    "2c";"e" \ar@{-};
  "2d";"g" \ar@{-};
"3a";"h"\ar@{-};
  "3b";"f"  \ar@{-};
"3a";"y"\ar@{-};
"y";"3b"\ar@{-};
"3a";"x"\ar@{-};
  "x";"b"  \ar@{-};
%
%
"3b";"z"\ar@{-};
  "z";"3c"  \ar@{-};
  "3c";"d"  \ar@{-};
 "3c";"v"  \ar@{-};
"v";"b"  \ar@{-};
     \endxy
$$
\end{small}

J. Stasheff points out that we can obtain the complex $\cal{K}(5)$ by deleting a single edge of $\cal{CK}(4).$

\subsection{Enriched bicategories and the composihedron}

Remarkably, the same minor error of recognition between the associahedron and composihedron may have been
  made by category theorists who wrote down the coherence
  axioms of enriched bicategory theory.  The first few polytopes in our new  sequence correspond to
   cocyle coherence conditions in the definition of enriched bicategories
 as in  \cite{Sean}. It is incorrectly implied there
   that
  the final cocycle condition has the combinatorial form of the associahedron $\cal{K}(5)$.
  The axiom actually consists of two pasting diagrams which when glued along their
   boundary are seen to form the composihedron $\cal{CK}(4)$ instead.
  Also note that the first few composihedra correspond as well to
 the axioms for pseudomonoids in a monoidal bicategory as seen in \cite{paddy}. This fact is to be expected, since
 pseudomonoids are just single object enriched bicategories.

Little has been published about enriched bicategories, although the theory is used in
recent research papers such as \cite{gurski} and \cite{shulman}. The full definition of enriched
bicategory is worked out in \cite{Sean}. It is repeated
with the simplification of a strict monoidal $\bcal{W}$ in \cite{Lack}, in which case the
commuting diagrams have the form of cubes.  (Recall that when both the range and domain are
 strictly associative that the multiplihedra collapse to
 become the cubes, as shown in \cite{BV1}.) An earlier  definition of
(lax) enriched bicategory as a lax functor of certain tricategories is found in \cite{GPS} and is also reviewed in \cite{Lack}; in retrospect this
formulation is to be expected since the composihedra are special cases of multiplihedra.

Here for reference is the definition of enriched bicategories, closely following \cite{Sean}.
Let $({\bcal W}, \otimes, \alpha, \pi, I)$ be a monoidal bicategory, as defined in \cite{GPS} or in \cite{Sean},
 or comparably
in \cite{Baez2},
with $I$ a strict unit (but we will include the identity cells in our diagrams).
Note that item (5) in the following definition corrects an obvious typo in the corresponding item of \cite{Sean},
 and that
the 2-cells in item (6) have been somewhat rearranged from that source, for easier comparison to the polytope
$\cal{CK}(4)$.
\begin{definition}
An enriched bicategory ${\bcal  A}$ over ${\bcal W}$  is:
\end{definition}
\begin{enumerate}
\item a collection of objects Ob${\bcal A},$
\item hom-objects ${\bcal A}(A,B) \in \text{Ob}{\bcal W}$ for each $A,B \in \text{Ob}{\bcal A}$,
\item composition 1-cells
$${\bcal M_{ABC}}:{\bcal A}(B,C)\otimes {\bcal A}(A,B) \to {\bcal A}(A,C)$$
 in ${\bcal W}$ for each $A,B,C \in \text{Ob}{\bcal A}$,
\item an identity 1-cell $\cal{J}_A:I\to\bcal{A}(A,A)$ for each object  $A,$
\item  2-cells ${\bcal M}_2$ in ${\bcal W}$ for each $A,B,C,D \in \text{Ob}{\cal A}$:
\noindent
                     \begin{center}
                 \resizebox{6in}{!}{
                 $$
                          \xymatrix@R=7pt@C=-5pt{
                           &({\bcal A}(C,D)\otimes {\bcal A}(B,C))\otimes {\bcal A}(A,B)
                           \ar[rr]^{\alpha}
                           \ar[ddl]^>>>>>>>>>>>>{{\bcal M} \otimes 1}
                           &~~~~~~~~~~~~~~~~~~~~~~~~~~~~~~~~~~~~~~&{\bcal A}(C,D)\otimes({\bcal A}(B,C)\otimes {\bcal A}(A,B))
                           \ar[ddr]_>>>>>>>>>>>>{1 \otimes {\bcal M}}&\\\\
                           {\bcal A}(B,D)\otimes {\bcal A}(A,B)
                           \ar[ddrr]_{{\bcal M}}
                           &\ar@{=>}[rr]^{{\bcal M}_2}&&&{\bcal A}(C,D)\otimes {\bcal A}(A,C)
                             \ar[ddll]^{{\bcal M}}
                             \\\\&&{\bcal A}(A,D)&&&
                        }$$ }\end{center}
\item ...which obey the following cocycle condition for each $A,B,C,D,E$.
In the following we abbreviate the hom-objects and composition by
${\bcal M}:BC,AB \to AC.$
\\\\
\begin{tiny}
$\xy 0;/r.75pc/:
   (0,16)*{DE,(CD,(BC,AB))}="a";
  (-11,11)*{(DE,CD),(BC,AB)}="b";
  (-16,0)*{((DE,CD),BC),AB}="c";
  (-11,-11)*{(CE,BC),AB}="d";
  (0,-16)*{BE,AB}="e";
  (11,-11)*{AE}="f";
  (16,0)*{DE,AD}="g";
  (11,11)*{DE,(CD,AC)}="h";
  (0,5)*{DE,((CD,BC),AB)}="2a";
  (-5,0)*{(DE,(CD,BC)),AB}="2b";
  (0,-5)*{(DE,BD),AB}="2c";
  (5,0)*{DE,(BD,AB)}="2d";
(7,7)*{{\bcal M}_2\Uparrow}="3a";
  (7,-7)*{{\bcal M}_2\Uparrow}="3b";
  (-7,-7)*{{\bcal M}_2\Uparrow}="3c";
  (-7,7)*{\pi\Uparrow};
  (0,2)*{=};
%
  "a";"b" \ar@{<-}; "a";"b"\ar@{<-};
  "b";"c"  \ar@{->};
    "c";"d" \ar@{->};
  "d";"e" \ar@{->};
 "e";"f" \ar@{<-};
 "f";"g" \ar@{<-};
 "g";"h" \ar@{<-};
 "h";"a" \ar@{<-};
 "2a";"2b"\ar@{->};
  "2b";"2c"  \ar@{->};
    "2c";"2d" \ar@{<-};
  "2d";"2a" \ar@{->};
"2a";"a"\ar@{<-};
  "2b";"c"  \ar@{->};
    "2c";"e" \ar@{->};
  "2d";"g" \ar@{-};
%
%
     \endxy$
\\
\txt{\Huge{=}}
\\
$\xy 0;/r.75pc/:
     (0,16)*{DE,(CD,(BC,AB))}="a";
  (-11,11)*{(DE,CD),(BC,AB)}="b";
  (-16,0)*{((DE,CD),BC),AB}="c";
  (-11,-11)*{(CE,BC),AB}="d";
  (0,-16)*{BE,AB}="e";
  (11,-11)*{AE}="f";
  (16,0)*{DE,AD}="g";
  (11,11)*{DE,(CD,AC)}="h";
  (0,5)*=0{}="2a";
  (-5,0)*=0{}="2b";
  (0,-5)*=0{}="2c";
  (5,0)*=0{}="2d";
(7,7)*{(DE,CD),AC}="3a";
  (7,-7)*{CE,AC}="3b";
  (-7,-7)*{CE,(BC,AB)}="3c";
    (11,0)*{{\bcal M}_2\Uparrow};
      (0,-11)*{{\bcal M}_2\Uparrow};
      (0,0)*{=};
      (-11,-4)*{=};
      (4,11)*{=};
%
 "a";"b" \ar@{<-}; "a";"b"\ar@{<-};
  "b";"c"  \ar@{->};
    "c";"d" \ar@{->};
  "d";"e" \ar@{->};
 "e";"f" \ar@{<-};
 "f";"g" \ar@{<-};
 "g";"h" \ar@{<-};
 "h";"a" \ar@{->};
"3a";"h"\ar@{->};
  "3b";"f"  \ar@{->};
"3a";"3b"\ar@{<-};
"3a";"b"  \ar@{<-};
%
%
"3b";"3c"  \ar@{<-};
  "3c";"d"  \ar@{<-};
 "3c";"b"  \ar@{-};
     \endxy$
\end{tiny}\\\\
\item  Unit 2-cells:
$$
                      \xymatrix@C=-5pt{
                       I\otimes {\bcal A}(A,B)
                      \ar[rrd]^{=}
                      \ar[dd]_{{\cal J}_{B}\otimes 1}
                      &&&&{\bcal A}(A,B)\otimes  I
                      \ar[dd]^{{1}\otimes {\cal J}_{A}}
                      \ar[lld]_{=}\\
                      &*{\lambda\Uparrow\hspace{.5in}}&{\bcal A}(A,B)&*{\hspace{.5in}\rho\Uparrow}&\\
                      {\bcal A}(B,B)\otimes {\bcal A}(A,B)
                      \ar[rru]|{{\bcal M}_{ABB}}
                      &&&& {\bcal A}(A,B)\otimes {\bcal A}(A,A)
                      \ar[llu]|{{\bcal M}_{AAB}}
                      }
              $$
\item ... which obey a pasting condition of their own, which we will omit for brevity.
\end{enumerate}

If instead of the existence of 2-cells postulated in (5) and (7) we had required that the
diagrams commute, we would recover the definition of an enriched category.
Note that the pentagonal axiom for a monoidal category  at the beginning of this
section has the form of the associahedron
$\cal{K}(4)$ but the commuting pentagon for enriched categories (here the domain and range of $\bcal{M}_2$),  is actually
better described as having the form of $\cal{CK}(3).$

Here is the cocycle condition (6), pasted together and shown as a Schlegel diagram for comparison
to the similarly displayed picture above of $\cal{CK}(4).$
To save space ``$\bullet~\bullet\to\bullet$'' will
                       represent
                       ${\bcal M}_1:{\bcal A}(B,C)\otimes {\bcal A}(A,B) \to {\bcal A}(A,C)$.
%
%
 \noindent
              \begin{center}
              \resizebox{4.25in}{!}{
              \begin{footnotesize}
  $$
  \xymatrix@R=18.5pt@C=2.5pt
  {&&((\bullet~\bullet)\bullet)\bullet\ar[rrrr]\ar[rd]\ar[llddd]&&&&(\bullet(\bullet~\bullet))\bullet\ar[rrddd]\ar[ld]
       \\
       &&&(\bullet~\bullet)\bullet\ar[rd]\ar[ld]&\buildrel{\bcal M}_2\over\Longrightarrow&(\bullet~\bullet)\bullet\ar[ld]\ar[rdd]
       \\
       && \bullet(\bullet~\bullet)\ar[rd]&{\bcal M}_2\Downarrow&
       \bullet~\bullet\ar[d]&\ar@{=>}[d]^{{\bcal M}_2}
       \\
       (\bullet~\bullet)(\bullet~\bullet)\ar[rru]\ar[rrd]\ar[rrrrdddd]&&& \bullet~\bullet\ar[r]& \bullet && \bullet(\bullet~\bullet)\ar[lld]&& \bullet((\bullet~\bullet)\bullet)\ar[ll]\ar[lllldddd]
       \\
       &&~~~~~~(\bullet~\bullet)\bullet\ar[ru]\ar[rrd]&{\bcal M}_2\Downarrow&
       \bullet~\bullet\ar[u]&\ar@{=>}[d]^{{\bcal M}_2}
       \\
       &&&& \bullet(\bullet~\bullet)\ar[u]&&&
       \\\\
       &&&& \bullet(\bullet(\bullet~\bullet)) \ar[uu]
   }
 $$
 \end{footnotesize}
                        }
                        \end{center}
From this perspective it is easier to visualize how, in the case
that $\bcal{W}=$ \textbf{Cat}, an enriched bicategory is just a bicategory.
In that case $\bcal{M}_2$ is renamed $\sigma,$ and the enriched cocycle axiom becomes the ordinary bicategory
 cocycle axiom with the
underlying form of $\cal{K}(4).$ Thus the $n^{th}$ composihedron
can be seen to contain in its structure two copies of the $n^{th}$ associahedron:
one copy as a particular upper facet and another as a certain collection of facets that can function as labels for the
facets of the corresponding associahedron. Of course the second copy is only seen upon decategorification!

\section{Recursive definition}\label{three}

In \cite{IM} the authors give a geometrically defined $CW$-complex definition of the multiplihedra,
and then demonstrate
the recursive combinatorial structure. Here
 we describe how to collapse that structure for the case of a strictly associative
domain, and achieve a recursive definition of the composihedra.

Pictures in the form of \emph{painted binary trees} can be drawn to represent the multiplication of
several objects in a monoid, before or after their passage to the image of that monoid under a
homomorphism. We use the term ``painted'' rather than ``colored'' to distinguish our trees with two
edge colorings, ``painted'' and ``unpainted,'' from the other meaning of colored, as in colored
operad or multicategory. We will refer to the exterior vertices of the tree as the  root and the
leaves , and to the interior vertices as nodes.  This will be handy since then we can reserve the
term ``vertices'' for reference to polytopes. A painted binary tree is painted beginning at
the root edge (the leaf edges are unpainted), and always painted in such a way that there are only three
types of nodes. They are:
$$
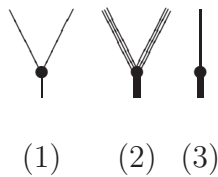
\xy 0;/r.25pc/:
  (-10,20)*=0{}="a"; (-2,20)*=0{}="b";
  (2,20)*{}="c"; (10,20)*{}="d";
  (14,20)*=0{}="e";
  (-6,12)*=0{\bullet}="v1"; (6,12)*=0{\bullet}="v2";
  (-6,4)*{\txt{\\\\(1)}}="v3"; (6,4)*{\txt{\\\\(2)}} ="v4";
  (14,12)*=0{\bullet}="v5"; (14,4)*{\txt{\\\\(3)}}="v6" \ar@{};
 "a"; "v1" \ar@{-};"a"; "v1" \ar@{-};
 "b"; "v1" \ar@{-};
 "v1"; "v3" \ar@{-}\ar@{};
 "c"; "v2" \ar@{-}; "c"; "v2" \ar@{-};
 "d"; "v2" \ar@{-};
 "v2"; "v4" \ar@{-} \ar@{};
 "c"; "v2" \ar@{=};"c"; "v2" \ar@{=};
 "d"; "v2" \ar@{=};
 "v2"; "v4" \ar@{=} \ar@{};
 "e"; "v5" \ar@{-} ;"e"; "v5" \ar@{-} \ar@{};
 "v5"; "v6" \ar@{-} ;"v5"; "v6" \ar@{-} \ar@{};
 "v5"; "v6" \ar@{=};"v5"; "v6" \ar@{=};
  \endxy
$$
This limitation on nodes implies that painted regions must be connected, that painting must never
end precisely at a trivalent node, and that painting must proceed up both branches of a trivalent
node. To see the promised representation we let the left-hand, type (1) trivalent node above stand
for multiplication in the domain; the middle, painted, type (2)  trivalent node above stand for
multiplication in the range; and the right-hand  type (3) bivalent  node stand for the action of
the mapping. For instance, given $a,b,c,d$ elements of a monoid, and $f$ a monoid morphism, the
following diagram represents the operation resulting in the product $f(ab)(f(c)f(d)).$
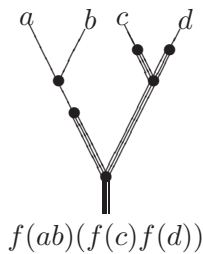
\begin{small}
 $$
\xy  0;/r.25pc/:
  (-10,20)*{a}="a"; (-2,20)*{b}="b";
  (2,20)*{c}="c"; (10,20)*{d}="d";
  (-6,12)*=0{\bullet}="v1"; (6,12)*=0{\bullet}="v2";
  (0,0)*=0{\bullet}="v3";
  (0,-7)*{\txt{\\$f(ab)(f(c)f(d))$}}="v4";
  (4,16)*=0{\bullet}="va";
  (8,16)*=0{\bullet}="vb";
  (-4,8)*=0{\bullet}="vc" \ar@{};
 "a" ;"vc" \ar@{-};"a" ;"vc" \ar@{-};
 "b" ;"v1" \ar@{-}  ;
 "c" ;"va" \ar@{-} ;
 "d" ;"vb" \ar@{-} \ar@{};
 "va";"v2" \ar@{=} ;"va";"v2" \ar@{=} ;
 "vb";"v3" \ar@{=} ;
 "vc" ;"v3" \ar@{=} ;
 "v3" ;"v4" \ar@{=} \ar@{};
 %
 %
 "va";"v2" \ar@{-} ;"va";"v2" \ar@{-} ;
 "vb";"v3" \ar@{-} ;
 "vc" ;"v3" \ar@{-} ;
 "v3" ;"v4" \ar@{-} ;
 \endxy
$$
\end{small}

To define the  face poset structures of the multiplihedra and composihedra we
 need painted trees that are no
longer binary. Here are the three new types of node allowed in a general painted tree. They
correspond to the the node types (1), (2) and (3) in that they are painted in similar fashion. They
generalize types (1), (2), and  (3) in that each has greater or equal valence than the
corresponding earlier node type.
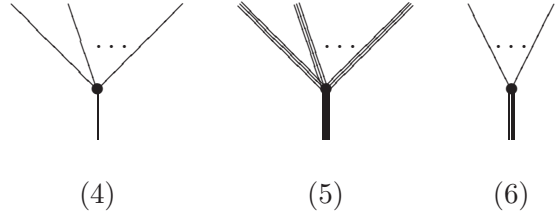
\begin{small}
 $$
\xy  0;/r.45pc/:
  (-14,20)*=0{}="a"; (-10,20)*=0{}="b";
  (-2,20)*=0{}="c"; (2,20)*=0{}="d";
  (6,20)*=0{}="e"; (14,20)*=0{}="f";
  (-7,17) *{\dots}; (9,17) *{\dots};
  (-8,14)*=0{\bullet}="v1"; (8,14)*=0{\bullet}="v2";
  (-8,8)*{\txt{\\\\(4)}}="v3";
  (8,8)*{\txt{\\\\(5)}}="v4"; \ar@{};
 (18,20)*=0{}="g"; (24,20)*=0{}="h";
 (21,17) *{\dots};
 (21,14)*=0{\bullet}="v5";
  (21,8)*{\txt{\\\\(6)}}="v6";
 "a" ;"v1" \ar@{-};"a" ;"v1" \ar@{-};
 "b" ;"v1" \ar@{-};
 "c" ;"v1" \ar@{-};
 "d" ;"v2" \ar@{-};
 "e" ;"v2" \ar@{-};
 "f" ;"v2" \ar@{-};
 "v1" ;"v3" \ar@{-};
 "v2" ;"v4" \ar@{-};
 "g" ;"v5" \ar@{-};
 "h" ;"v5" \ar@{-};
 "v5" ;"v6" \ar@{-};
 \ar@{} ;
"d" ;"v2" \ar@{=};"d" ;"v2" \ar@{=};
 "e" ;"v2" \ar@{=};
 "f" ;"v2" \ar@{=};
 "v2" ;"v4" \ar@{=};
 "v5" ;"v6" \ar@{=};
 \endxy
$$
\end{small}
\begin{definition}
By \emph{refinement} of painted trees we refer to the
relationship: $t$ refines $t'$ means that $t'$ results from the collapse of some of the internal edges of
$t$.  This is a partial order on
$n$-leaved painted trees, and we write $t < t'.$
Thus the binary painted trees are refinements of the trees having nodes of type (4)-(6).
\emph{Minimal refinement} refers to
the following specific case of refinement:
 $t$ minimally refines $t''$ means that $t$ refines $t''$ and also that there is no
 $t'$ such that both $t$ refines $t'$ and $t'$ refines $t''$.
\end{definition}

The poset of painted trees with $n$ leaves is precisely the face poset of the $n^{th}$ multiplihedron.

\begin{definition}
Two painted trees are said to be \emph{domain equivalent} if two requirements are satisfied:
 1) they both refine the same tree, and 2)
 the collapses
 involved in both refinements are of edges whose two nodes are of type (1) or type (4). That is, these
  collapses will be of
 internal unpainted edges with no adjacent painted edges. Locally the equivalences will appear as follows:
 
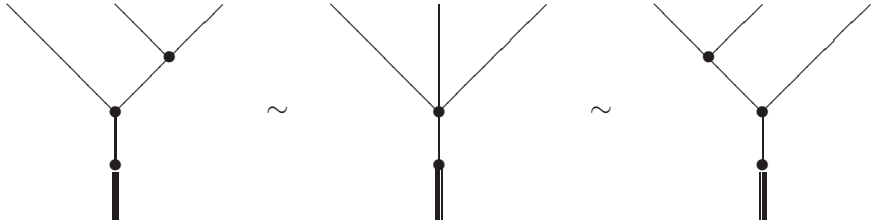
\begin{small}
$$
\xy  0;/r.85pc/:
  (-16,8)*=0{}="a"; (-12,8)*=0{}="b";
  (-8,8)*=0{}="c"; (-4,8)*=0{}="d";
  (0,8)*=0{}="e"; (4,8)*=0{}="f";
  (8,8)*=0{}="g"; (12,8)*=0{}="h";
  (16,8)*=0{}="i";
  (-6,4) *{\sim}; (6,4) *{\sim};
  (-10,6)*=0{\bullet}="v1"; (-12,4)*=0{\bullet}="v2";
  (-12,2)*{\bullet}="v3";
  (-12,0)*{}="v4";
  (10,6)*=0{\bullet}="v6"; (12,4)*=0{\bullet}="v7";
  (12,2)*{\bullet}="v8";
  (12,0)*{}="v9";
  (0,4)*=0{\bullet}="va"; (0,0)*=0{}="vb";
  (0,2)*=0{\bullet}="v10";
   \ar@{};
  "a" ;"v2" \ar@{-};"a" ;"v2" \ar@{-};
 "b" ;"v1" \ar@{-};
 "c" ;"v1" \ar@{-};
 "v1" ;"v2" \ar@{-};
 "v2" ;"v3" \ar@{-};
 "v3" ;"v4" \ar@{-};
 "d" ;"va" \ar@{-};
 "e" ;"va" \ar@{-};
 "f" ;"va" \ar@{-};
 "va" ;"vb" \ar@{-};
 "g" ;"v6" \ar@{-};
 "h" ;"v6" \ar@{-};
 "i" ;"v7" \ar@{-};
 "v6" ;"v7" \ar@{-};
 "v7" ;"v8" \ar@{-};
 "v8" ;"v9" \ar@{-};
 "v3" ;"v4" \ar@{=};"v3" ;"v4" \ar@{=};
 "v10" ;"vb" \ar@{=};
 "v8" ;"v9" \ar@{=};
 \endxy
$$
\end{small}
That is, the domain equivalence is generated by the two moves illustrated above.
Therefore we often choose to represent a domain
equivalence class of trees by the unique member with least refinement, that is with
its unpainted subtrees all corollas.
\end{definition}
\begin{definition}
A  \emph{painted
corolla} is
 a painted tree with only one node, of type (6).
\end{definition}

The $n^{th}$ composihedron may be described as the quotient of the  $n^{th}$ multiplihedron under domain equivalence.
This quotient can be performed by applying the equivalence either
to the metric trees which define the multiplihedra in \cite{BV1} or to the combinatorial definition of the multiplihedra,
i.e. to the face poset of the multiplihedra. Here we will follow the latter scheme to unpack the definition of the composihedra
into a recursive description if facet inclusions.

The facets of the $n^{th}$ multiplihedron are of two types: upper facets
whose vertices correspond to sets of related ways of multiplying in the range, and lower facets whose vertices
correspond to sets of related ways of multiplying in the domain. A lower facet is denoted
  ${\cal J}_k(r,s)$ and is a combinatorial copy of
the complex ${\cal J}(r) \times{\cal K}(s).$
Here $r+s-1 = n.$ A vertex of a lower facet represents a way in which $s$ of the points are multiplied in the
domain, and then how the images of their product and of the other $r-1$ points are multiplied in the range.
In the case for which the domain is strictly associative, the copy of
${\cal K}(s)$ here should be replaced by a single point, denoted $\{*\}.$  Thus in the composihedra the lower faces
will have reduced dimension. In fact only certain of them will still be facets.

The recursive definition of the new sequence of polytopes is as follows:
\begin{definition}\label{recur}
The first composihedron denoted $\cal{CK}(1)$ is defined to be the single point $\{*\}.$
It is associated to the painted tree with one leaf, and thus one type (3) internal node.
Assume that the $\cal{CK}(k)$ have been defined for $k=1\dots n-1.$ To $\cal{CK}(k)$
we associate the $k$-leaved painted corolla.
We define an $(n-2)$-dimensional
 $CW$-complex $\partial\cal{CK}(n)$
as follows, and then define $\cal{CK}(n)$ to be the cone on $\partial\cal{CK}(n)$.  Now the top-dimensional
cells of $\partial\cal{CK}(n)$ (facets of $\cal{CK}(n)$) are in
 bijection with the set of  painted trees of two types:
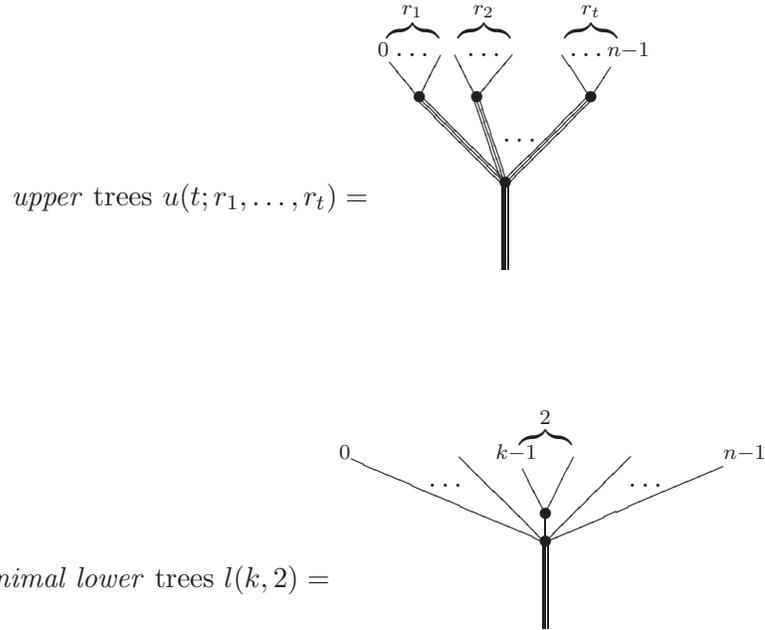
\begin{small}
 $$
\xy (0,0) *{ \text{\emph{upper} trees $u(t;r_1,\dots, r_t) =$ }} \endxy
\xy (0,0) *{
\xy  0;/r.45pc/:
  (-.5,23)*{{}^0}="a"; (1.5,25.25) *{{r_1 \atop \overbrace{~~~}}}; (1.5,23) *{\dots}; (3.5,23)*=0{}="b";
(4.5,23)*=0{}="a2"; (6.5,25.25) *{{r_2 \atop \overbrace{~~~}}}; (6.5,23)  *{\dots};
(8.5,23)*=0{}="b2";
 (12,23)*=0{}="a3"; (14,25.25) *{{r_t \atop \overbrace{~~~}}}; (14,23) *{\dots~~}; (16,23)*{~~^{~n-1}}="b3";
   (2,20)*=0{\bullet}="d";
  (6,20)*=0{\bullet}="e"; (14,20)*=0{\bullet}="f";
 (9,17) *{\dots};
  (8,14)*=0{\bullet}="v2";
  (8,8)*=0{}="v4"; \ar@{};
 "d" ;"v2" \ar@{-}; "d" ;"v2" \ar@{-};
"a" ;"d" \ar@{-};
 "b" ;"d" \ar@{-};
  "a2" ;"e" \ar@{-};
   "b2" ;"e" \ar@{-};
 "a3" ;"f" \ar@{-};
 "b3" ;"f" \ar@{-};
 "e" ;"v2" \ar@{-};
 "f" ;"v2" \ar@{-};
  "v2" ;"v4" \ar@{-};
 \ar@{} ;
"d" ;"v2" \ar@{=};"d" ;"v2" \ar@{=};
 "e" ;"v2" \ar@{=};
 "f" ;"v2" \ar@{=};
 "v2" ;"v4" \ar@{=};
 \endxy }\endxy
 $$
 $$
\xy (0,0) *{ \text{ and \emph{minimal lower} trees $l(k,2) =$ }}\endxy
\xy (0,0) *{
\xy  0;/r.45pc/:
    (-8,22) *{{2 \atop \overbrace{~~~~~~~~~}}};
   (-22,20)*{{}^0}="L";(6,20)*{{}^{n-1}}="R";
  (-14,20)*=0{}="a"; (-10,20)*{{}^{k-1}}="b";
  (-6,20)*=0{}="b2";
  (-2,20)*=0{}="c";
  (-8,14)*=0{\bullet}="v1";
  (-8,8)*=0{}="v3";
  (-8,16)*=0{\bullet}="v0";
  (-15,18) *{\dots};
  (-1,18) *{\dots};
  %
   "a" ;"v1" \ar@{-};"a" ;"v1" \ar@{-};
   "L" ;"v1" \ar@{-};
   "R" ;"v1" \ar@{-};
 "b" ; "v0" \ar@{-};
 "b2" ; "v0" \ar@{-};
 "v0"; "v1" \ar@{-};
 "v1" ; "c" \ar@{-};
  "v1" ;"v3" \ar@{-};
  \ar@{} ;
"v1" ;"v3" \ar@{=};"v1" ;"v3" \ar@{=};
 \endxy } \endxy
$$
\end{small}
There are $2^{n-1}-1$ upper trees, as counted in \cite{multi}.
The index $t$ can range from $2\dots n,$ and $r_i \ge 1.$
The upper facet which we call $\cal{CK}(t;,r_1,\dots,r_t)$ associated with the upper tree
 of identical indexing is a copy of $\cal{K}(t) \times \cal{CK}(r_1)\times \dots\times\cal{CK}(r_t).$
 Here the sum of the $r_i$ is $n.$

There are $n-1$ minimal lower trees. The lower facet which we call $\cal{CK}_k(n-1,2)$ associated with
 the lower tree of identical indexing is a copy of $\cal{CK}(n-1)\times \cal{K}(2)$, that is, a copy of $\cal{CK}(n-1).$  Here
 $k$ ranges from $1$ to $n-1.$
 $B(n)$ is the union of all these facets, with intersections described as follows:

Since the facets are product polytopes, their sub-facets in turn are
products of faces (of smaller associahedra and composihedra) whose dimensions sum to $n-3.$ Each of these
sub-facets thus comes (inductively) with a list of associated trees.
There will always be a unique way of grafting these trees to
construct a painted tree that is a minimal refinement of the upper or minimal
 lower tree associated to the facet in question.
For the sub-facets
of an upper facet the recipe is to paint entirely the $t$-leaved tree associated to a face of $\cal{K}(t)$
and to graft to each of its branches in turn the trees associated to the appropriate faces of  $\cal{CK}(r_1)$ through
$\cal{CK}(r_t)$ respectively.
 A sub-facet of the
 lower facet $\cal{CK}_k(n-1,2)$ inductively comes with an $(n-1)$-leaved associated upper or minimal lower tree $T$.
 The recipe
 for assigning our sub-facet
  an $n$-leaved
 minimal refinement of the $n$-leaved minimal lower tree $l(k,2)$ is to graft an unpainted 2-leaved tree to the
 $k^{th}$
 leaf of $T.$ See the following Example~\ref{ex1} for this pictured in low dimensions.

 The intersection of two facets in $\partial\cal{CK}(n)$ consists of the sub-facets of each which
have associated trees that are domain equivalent.
 Then $\cal{CK}(n)$ is defined to be the cone on $\partial\cal{CK}(n).$
 To $\cal{CK}(n)$ we assign the painted corolla of $n$ leaves.
\end{definition}

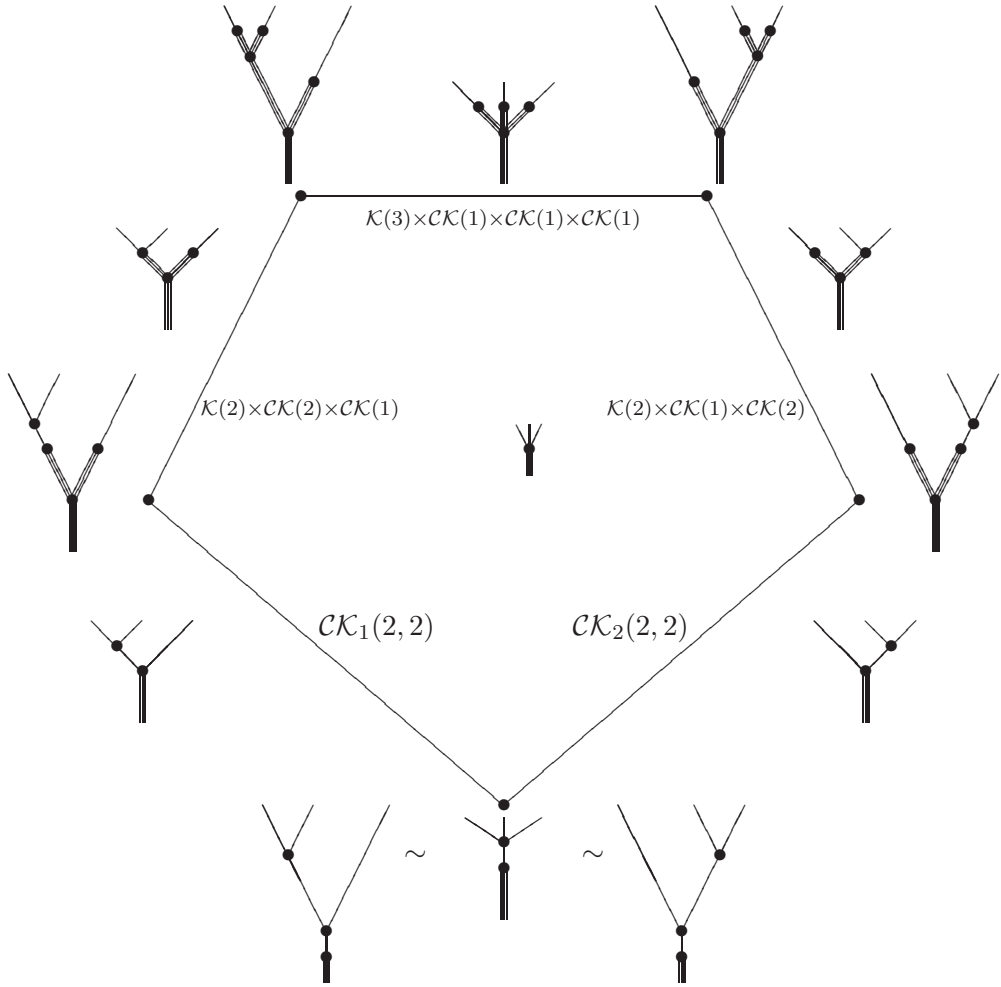
\begin{example}\label{ex1}

$$
\cal{CK}(1) = \bullet~\hspace{.05in}
\xy 0;/r.25pc/:
    (14,6)*=0{}="e";
  (14,2)*=0{\bullet}="v5"; (14,-3)*=0{}="v6" \ar@{};
  "e"; "v5" \ar@{-} ;"e"; "v5" \ar@{-} \ar@{};
 "v5"; "v6" \ar@{-} ;"v5"; "v6" \ar@{-} \ar@{};
 "v5"; "v6" \ar@{=};"v5"; "v6" \ar@{=};
  \endxy
$$

Here is $\cal{CK}(2)$ with the upper facet $\cal{K}(2) \times \cal{CK}(1)\times \cal{CK}(1)$ on the
left and the lower facet $\cal{CK}(1)\times \cal{K}(2)$ on the right.

\begin{small}
$$
\xy 0;/r.55pc/: (-12,8)*=0{}="a"; (-4,8)*=0{}="b";
  (4,8)*=0{}="e"; (12,8)*=0{}="f";
  (-2,4)*=0{}="c"; (2,4)*=0{}="d";
(-10,6)*=0{\bullet}="va";
 (-6,6)*=0{\bullet}="vb";
(-8,4)*=0{\bullet}="v1";
 (0,2)*=0{\bullet}="v2";
  (8,4)*=0{\bullet}="v3";
  (8,2)*=0{\bullet}="v4";
  (-8,0)*=0{}="r1";
(0,0)*=0{}="r2";
 (8,0)*=0{}="r3";
 (-8,-1)*=0{\bullet}="lp";
 (0,-1)*=0{}="cp";
 (8,-1)*=0{\bullet}="rp";
"a" ;"v1" \ar@{-}; "a" ;"v1" \ar@{-};
 "b" ;"v1" \ar@{-};
 "v1" ;"r1" \ar@{-};
 "c" ;"v2" \ar@{-};
 "d" ;"v2" \ar@{-};
 "v2" ;"r2" \ar@{-};
 "e" ;"v3" \ar@{-};
 "f" ;"v3" \ar@{-};
 "v3" ;"v4" \ar@{-};
 "v4" ;"r3" \ar@{-};
 "lp" ;"rp" \ar@{-};
 "va" ;"v1" \ar@{=}; "va" ;"v1" \ar@{=};
 "vb" ;"v1" \ar@{=};
 "v1" ;"r1" \ar@{=};
 "v2" ;"r2" \ar@{=};
 "v4" ;"r3" \ar@{=};
 \endxy
$$
\end{small}

$$$$
$$$$

And here is the complex $\cal{CK}(3).$ The product structure of the upper facets is listed, and the
two lower facets are named. Note that both lower facets are copies of $\cal{CK}(2)\times\cal{K}(2).$
Notice also how  the sub-facets (vertices) are labeled. For instance, the upper right vertex is labeled by a tree
that could be constructed by grafting three copies of the single leaf painted corolla onto a completely painted binary tree
with three leaves, or by grafting a single leaf painted corolla and a 2-leaf painted corolla onto the leaves of a 2-leaf
completely painted binary tree.

\begin{small}
$$
\xy 0;/r.2pc/:
 (-32,48);(0,48) *=0{}  \ar@{-}; (-32,48);(0,48) *=0{}  \ar@{-};
 (0,44) *={\txt{$^{\cal{K}(3)\times\cal{CK}(1)\times\cal{CK}(1)\times\cal{CK}(1)}$}};
 (0,48);(32,48)*=0{\bullet}  \ar@{-};
   (32,48);(44,24)*=0{}  \ar@{-};
   (32,14) *={^{\cal{K}(2)\times\cal{CK}(1)\times\cal{CK}(2)}};
   (-32,14) *={^{\cal{K}(2)\times\cal{CK}(2)\times\cal{CK}(1)}};
  (44,24); (56,0)*=0{\bullet}  \ar@{-};
   (56,0); (0,-48)*=0{\bullet}  \ar@{-};
   (20,-20) *={\cal{CK}_2(2,2)};
   (-20,-20) *={\cal{CK}_1(2,2)};
     (0,-48); (-56,0)*=0{\bullet}  \ar@{-};
      (-56,0);(-44,24)*=0{}  \ar@{-};
      (-44,24); (-32,48)*=0{\bullet}  \ar@{=};
            (4,4);(4,8) \ar@{=}; (4,4);(4,8) *=0{\bullet}\ar@{-}; (4,4);(4,8) \ar@{-};
      (4,8);(2,12)  \ar@{-}; (4,8); (4,12) \ar@{-}; (4,8); (6,12) \ar@{-};
(   -44 ,   78  )      *=0{}="aa";
 (   -36 ,   78  )      *=0{}="ab";
  (   -32 ,   78  ) *=0{}="ac";
 (   -24 ,   78  )      *=0{}="ad";
  (   -40 ,   70  )      *=0{\bullet}="av1";
   (   -28 ,70  )   *=0{}="av2";
 (   -34 ,   58  )      *=0{\bullet}="av3";
  (   -34 ,   50  )      *=0{}="av4";
(   -30 ,   74  )      *=0{}="ava";
 (   -26 ,   74  )      *=0{}="avb";
  (   -38 ,   74  )*=0{\bullet}="ave";
 (   -42 ,   74  )      *=0{\bullet}="avf";
  (   -38 ,   66  )      *=0{}="avc";
(   -30 ,   66  )   *=0{\bullet}="avd";
 "aa" ;"avc" \ar@{-};"aa" ;"avc" \ar@{-};
 "ab" ;"av1" \ar@{-}  ;
  "ad" ;"avb" \ar@{-} \ar@{};
 "ave";"av1" \ar@{=};"ave";"av1" \ar@{=} ;
 "avf";"av1" \ar@{=} ;
 "av1" ;"av3" \ar@{=} ;
 "avd" ;"av3" \ar@{=} ;
 "av3" ;"av4" \ar@{=} \ar@{};
"avb";"av3" \ar@{-} ;
 "avb";"av3" \ar@{-} ;
 "avc" ;"av3" \ar@{-} ;
 "av3" ;"av4" \ar@{-} ;
 (   24  ,   78  )   *=0{}="ba";
(   32  ,   78  )   *=0{}="bb";
 (   36  ,   78  )   *=0{}="bc";
  (   44  ,   78  )   *=0{}="bd";
   (28  ,   70  )   *=0{}="bv1";
 (   40  ,   70  )   *=0{\bullet}="bv2";
  (   34  ,   58  )*=0{\bullet}="bv3";
 (   34  ,   50  )   *=0{}="bv4";
  (   38  ,   74  )   *=0{\bullet}="bva";
   (   42,   74  )   *=0{\bullet}="bvb";
 (   30  ,   74  )   *=0{}="bve";
  (   26  ,   74  )   *=0{}="bvf";
   (30  ,   66  )   *=0{\bullet}="bvc" ;
 (   38  ,   66  )   *=0{}="bvd" ;
 "ba" ;"bvc" \ar@{-};"ba" ;"bvc" \ar@{-};
  "bc" ;"bva" \ar@{-} ;
 "bd" ;"bvb" \ar@{-} \ar@{};
 "bva";"bv2" \ar@{=} ;"bva";"bv2" \ar@{=} ;
 "bvb";"bv2" \ar@{=} ;
 "bv2";"bv3" \ar@{=} ;
 "bvc" ;"bv3" \ar@{=} ;
 "bv3" ;"bv4" \ar@{=} \ar@{};
 "bva";"bv2" \ar@{-} ;"bva";"bv2" \ar@{-} ;
 "bvb";"bv3" \ar@{-} ;
 "bvc" ;"bv3" \ar@{-} ;
 "bv3" ;"bv4" \ar@{-} ;
(   -78 ,   20  )            *=0{}="ca";
 (   -70 ,   20  )            *=0{}="cb";
  (   -66 ,   20  )*=0{}="cc";
 (   -58 ,   20  )            *=0{}="cd";
  (   -74 ,   12  )*=0{\bullet}="cv1";
 (   -62 ,   12  )            *=0{}="cv2";
  (   -68 ,   0   )*=0{\bullet}="cv3";
 (   -68 ,   -8  )   *=0{}="cv4";
  (   -64 ,   16  )            *=0{}="cva";
   (-60 ,   16  )            *=0{}="cvb";
 (   -72 ,   16  )      *=0{}="cve";
  (   -76 ,   16  )*=0{}="cvf";
 (   -72 ,   8   )   *=0{\bullet}="cvc" ;
  (   -64 ,   8   )*=0{\bullet}="cvd" ;
 "ca" ;"cv1" \ar@{-};"ca" ;"cv1" \ar@{-};
 "cv1" ;"cvc" \ar@{-};
 "cb" ;"cv1" \ar@{-}  ;
  "cd" ;"cvb" \ar@{-} \ar@{};
 "cvc" ;"cv3" \ar@{=} ;
 "cvc" ;"cv3" \ar@{=} ;
 "cvd" ;"cv3" \ar@{=} ;
 "cv3" ;"cv4" \ar@{=} \ar@{};
"cvb";"cv3" \ar@{-} ;
 "cvb";"cv3" \ar@{-} ;
 "cvc" ;"cv3" \ar@{-} ;
 "cv3" ;"cv4" \ar@{-} ;
(   58  ,   20  )      *=0{}="da";
 (   66  ,   20  )      *=0{}="db";
  (   70  ,   20  )*=0{}="dc";
 (   78  ,   20  )      *=0{}="dd";
  (   62  ,   12  )      *=0{}="dv1";
   (   74  ,   12)   *=0{\bullet}="dv2";
 (   68  ,   0   )      *=0{\bullet}="dv3";
  (   68  ,   -8  )*=0{}="dv4";
 (   72  ,   16  )      *=0{}="dva";
  (   76  ,   16  )      *=0{}="dvb";
   (   64  ,   16)      *=0{}="dve";
 (   60  ,   16  )      *=0{}="dvf";
  (   64  ,   8   )      *=0{\bullet}="dvc" ;
(   72  ,   8   )   *=0{\bullet}="dvd" ;
 "da" ;"dvc" \ar@{-};"da" ;"dvc" \ar@{-};
  "dc" ;"dva" \ar@{-} ;
 "dd" ;"dvb" \ar@{-} \ar@{};
"dvd" ;"dv3" \ar@{=} ;"dvd" ;"dv3" \ar@{=} ;
 "dvc" ;"dv3" \ar@{=} ;
 "dv3" ;"dv4" \ar@{=} \ar@{};
 "dva";"dv2" \ar@{-} ;"dva";"dv2" \ar@{-} ;
 "dvb";"dv2" \ar@{-} ;
 "dv2";"dvd" \ar@{-} ;
  "dvd";"dv3" \ar@{-} ;
 "dvc" ;"dv3" \ar@{-} ;
 "dv3" ;"dv4" \ar@{-} ;
(   -38 ,   -48 )   *=0{}="ea";
 (   -30 ,   -48 )   *=0{}="eb";
  (   -26 ,   -48 )   *=0{}="ec";
   (-18 ,   -48 )   *=0{}="ed";
 (   -34 ,   -56 )   *=0{\bullet}="ev1";
  (   -22 ,   -56 )*=0{}="ev2";
 (   -28 ,   -68 )   *=0{\bullet}="ev3";
  (   -28 ,   -76 )   *=0{}="ev4";
   (   -28 ,-72 )   *=0{\bullet}="ev5";
 (   -24 ,   -52 )   *=0{}="eva";
  (   -20 ,   -52 )   *=0{}="evb";
   (-32 ,   -52 )   *=0{}="eve";
 (   -36 ,   -52 )   *=0{}="evf";
  (   -32 ,   -60 )   *=0{}="evc" ;
   (-24 ,   -60 )   *=0{}="evd" ;
(-14,-56) *={\sim};
 "ea" ;"evc" \ar@{-};"ea" ;"evc" \ar@{-};
 "eb" ;"ev1" \ar@{-}  ;
  "ed" ;"evb" \ar@{-} \ar@{};
   "ev5" ;"ev4" \ar@{=}; "ev5" ;"ev4" \ar@{=} \ar@{};
"evb";"ev3" \ar@{-} ;
 "evb";"ev3" \ar@{-} ;
 "evc" ;"ev1" \ar@{-} ;
 "ev1" ;"ev3" \ar@{-} ;
 "ev3" ;"ev5" \ar@{-} ;
 "ev5" ;"ev4" \ar@{-} ;
  (   18  ,   -48 )   *=0{}="fa";
(   26  ,   -48 )   *=0{}="fb";
 (   30  ,   -48 )   *=0{}="fc";
  (   38  ,   -48 )   *=0{}="fd";
   (22  ,   -56 )   *=0{}="fv1";
 (   34  ,   -56 )   *=0{\bullet}="fv2";
  (   28  ,   -68 )*=0{\bullet}="fv3";
 (   28  ,   -76 )   *=0{}="fv4";
  (   28  ,   -72 )   *=0{\bullet}="fv5";
   (   32,   -52 )   *=0{}="fva";
 (   36  ,   -52 )   *=0{}="fvb";
  (   24  ,   -52 )   *=0{}="fve";
   (   20,   -52 )   *=0{}="fvf";
 (   24  ,   -60 )   *=0{}="fvc" ;
  (   32  ,   -60 )   *=0{}="fvd" ;
(14,-56) *={\sim};
 "fa" ;"fvc" \ar@{-};"fa" ;"fvc" \ar@{-};
  "fc" ;"fva" \ar@{-} ;
 "fd" ;"fvb" \ar@{-} \ar@{};
  "fv5" ;"fv4" \ar@{=}; "fv5" ;"fv4" \ar@{=} \ar@{};
 "fva";"fv2" \ar@{-} ;"fva";"fv2" \ar@{-} ;
  "fvb";"fv2" \ar@{-} ;
 "fv2";"fv3" \ar@{-} ;
 "fvc" ;"fv1" \ar@{-} ;
 "fv1" ;"fv3" \ar@{-} ;
 "fv3" ;"fv5" \ar@{-} ;
 "fv5" ;"fv4" \ar@{-} ;
 (-8,66)*=0{}="a"; (8,66)*=0{}="b";
 (0,66)*=0{}="c";
  (-4,62)*=0{\bullet}="va";
  (0,62)*=0{\bullet}="vc";
 (4,62)*=0{\bullet}="vb";
(0,58)*=0{\bullet}="v1";
  (0,50)*=0{}="r1";
"a" ;"v1" \ar@{-}; "a" ;"v1" \ar@{-};
  "c" ; "v1" \ar@{-};
 "b" ;"v1" \ar@{-};
 "v1" ;"r1" \ar@{-};
 "va" ;"v1" \ar@{=}; "va" ;"v1" \ar@{=};
 "vb" ;"v1" \ar@{=};
 "vc" ; "v1" \ar@{=};
 "v1" ;"r1" \ar@{-};
(45,43)*=0{}="a"; (61,43)*=0{}="b";
 (53,43)*=0{}="c";
  (49,39)*=0{\bullet}="va";
 (57,39)*=0{\bullet}="vb";
(53,35)*=0{\bullet}="v1";
  (53,27)*=0{}="r1";
"a" ;"v1" \ar@{-}; "a" ;"v1" \ar@{-};
  "c" ; "vb" \ar@{-};
 "b" ;"v1" \ar@{-};
 "v1" ;"r1" \ar@{-};
 "va" ;"v1" \ar@{=}; "va" ;"v1" \ar@{=};
 "vb" ;"v1" \ar@{=};
 "v1" ;"r1" \ar@{-};
(-45,43)*=0{}="a";
 (-61,43)*=0{}="b";
 (-53,43)*=0{}="c";
  (-49,39)*=0{\bullet}="va";
 (-57,39)*=0{\bullet}="vb";
(-53,35)*=0{\bullet}="v1";
  (-53,27)*=0{}="r1";
"a" ;"v1" \ar@{-}; "a" ;"v1" \ar@{-};
  "c" ; "vb" \ar@{-};
 "b" ;"v1" \ar@{-};
 "v1" ;"r1" \ar@{-};
 "va" ;"v1" \ar@{=}; "va" ;"v1" \ar@{=};
 "vb" ;"v1" \ar@{=};
 "v1" ;"r1" \ar@{-};
 (-6,-50)*=0{}="a"; (6,-50)*=0{}="b";
 (0,-50)*=0{}="c";
    (0,-54)*=0{\bullet}="vc";
 (0,-58)*=0{\bullet}="v1";
  (0,-66)*=0{}="r1";
"a" ;"vc" \ar@{-}; "a" ;"vc" \ar@{-};
  "c" ; "vc" \ar@{-};
 "b" ;"vc" \ar@{-};
 "v1" ;"r1" \ar@{-};
  \ar@{-};
 "vc" ; "v1" \ar@{=};
 "v1" ;"r1" \ar@{-};
(49,-19)*=0{}="a"; (65,-19)*=0{}="b";
 (57,-19)*=0{}="c";
  (53,-23)*=0{}="va";
   (61,-23)*=0{\bullet}="vb";
(57,-27)*=0{\bullet}="v1";
  (57,-35)*=0{}="r1";
"a" ;"v1" \ar@{-}; "a" ;"v1" \ar@{-};
  "c" ; "vb" \ar@{-};
 "b" ;"v1" \ar@{-};
 "v1" ;"r1" \ar@{-};
 "va" ;"v1" \ar@{-}; "va" ;"v1" \ar@{-};
 "vb" ;"v1" \ar@{=};
  "v1" ;"r1" \ar@{-};
(-49,-19)*=0{}="a"; (-65,-19)*=0{}="b";
 (-57,-19)*=0{}="c";
  (-53,-23)*=0{}="va";
   (-61,-23)*=0{\bullet}="vb";
(-57,-27)*=0{\bullet}="v1";
  (-57,-35)*=0{}="r1";
"a" ;"v1" \ar@{-}; "a" ;"v1" \ar@{-};
  "c" ; "vb" \ar@{-};
 "b" ;"v1" \ar@{-};
 "v1" ;"r1" \ar@{-};
 "va" ;"v1" \ar@{-}; "va" ;"v1" \ar@{-};
 "vb" ;"v1" \ar@{=};
  "v1" ;"r1" \ar@{-};
 \endxy
$$
\end{small}
\end{example}

\begin{remark}
The total number of facets of the $n^{th}$ composihedron is thus $2^{n-1}+n-2.$ Therefore the number
of facets of $\cal{CK}(n+1)$ is $2^n+n-1 = 0,2,5,10,19,36,\dots.$  This is sequence A052944 in the OEIS.
This sequence also gives the number of vertices of a $n$-dimensional cube with one truncated corner.
Now the relationship between the polytopes in Figure 1 can be algebraically stated as an equation:
$$
\left(2^n+n-1\right)+\left(\frac{(n+1)(n+2)}{2}-1\right)-\left(\frac{n(n+1)}{2}+2^n-1\right) = 2n
$$
That is:
$$
|\{\text{ facets of } \cal{CK}(n+1) \}| +|\{\text{ facets of } \cal{K}(n+2) \}|-|\{\text{ facets of } \cal{J}(n+1) \}|=|\{\text{ facets of }C(n) \}|
$$
where $C(n)$ is the $n$-dimensional cube.
\end{remark}
\begin{remark}
                 Consider the existence of facet inclusions:
$\cal{K}(k) \times (\cal{CK}(j_1) \times \dots \times\cal{CK}(j_k)) \to \cal{CK}(n)$
where $n$ is the sum of the $j_i$. This is the left module structure: the composihedra together form a left operad module
over the operad of spaces formed by the associahedra. Notice that the
description of the assignation of trees to sub-facets (via grafting)  in the definition above guarantees the
module axioms.
\end{remark}

\begin{remark}
By definition, the $n^{th}$ composihedron is seen to be the quotient of the $n^{th}$
 multiplihedron under domain equivalence. Recall that the lower facets of $\cal{J}(n)$ are
 equivalent to copies of $\cal{J}(r)\times\cal{K}(s).$
 The actual quotient is achieved by identifying any two points $(a,b)\sim (a,c)$ in a lower facet,
 where $a$ is a point of $\cal{J}(r)$ and $b,c$ are points in $\cal{K}(s)$.
An alternate definition of the composihedra describes
$\cal{CK}(n)$ as the polytope achieved by taking the $n^{th}$
multiplihedron  ${\cal J}(n)$  and
sequentially collapsing certain facets.
 Any cell in ${\cal J}(n)$ whose 1-skeleton is entirely made up of edges that correspond
to images of domain associations  $f((ab)c)\to f(a(bc))$ is collapsed to a single vertex.  In the resulting complex
there will be redundant cells. First any cell with exactly two
vertices will be collapsed to an edge between them. Then any cell with
1-skeleton $S^1$ will be collapsed to a 2-disk spanning that boundary. This
process continues inductively until any cell with $(n-3)$-skeleton $S^{(n-3)}$
is collapsed to  a $(n-2)$-disk spanning that boundary.
For a picture of this see the upper left projection of Figure 1, where the small pentagon and two rectangles on the back
of  $\cal{J}(4)$ collapse to a vertex and two edges respectively.
\end{remark}

\begin{remark}
  In \cite{umble} there is described a projection $\pi$ from the $n^{th}$
 permutohedron $\cal{P}(n)$ to $\cal{J}(n),$ and so there follows a composite projection from
 $\cal{P}(n)\to \cal{CK}(n).$ For comparison sake, recall that in the case of a strictly associative range the
 multiplihedra become the associahedra, in fact the quotient under range equivalence of $\cal{J}(n)$ is $\cal{K}(n+1).$
 The implied projection of this quotient composed with $\pi$ yields a new projection from $\cal{P}(n)\to \cal{K}(n+1).$
 Saneblidze and Umble describe a very different projection from $\cal{J}(n)$ to $\cal{K}(n+1).$  When
 composed with $\pi$ this yields a (different) projection from $\cal{P}(n)$ to $\cal{K}(n+1)$ that
  is  shown in \cite{umble} to be precisely the
 projection $\theta$ described in \cite{Tonks}.
\end{remark}

\begin{remark}
There is a bijection between
the $0$-cells, or vertices, of $\cal{CK}(n)$ and the domain equivalence classes of $n$-leaved painted binary trees.
This follows from the recursive construction, since the 0-cells must be associated to completely refined painted trees.
However, the domain equivalent trees must be assigned to the same vertex since domain equivalence determines
intersection of
sub-facets.
Recall that we can also label the vertices of $\cal{CK}(n)$ by applications of an $A_{\infty}$ function $f$ to
$n$-fold products in a topological monoid.
\end{remark}

\begin{remark}
As defined in \cite{sta2} an $A_n$ map between $A_n$ spaces $f:X \to Y$ is described by an action
$\cal{J}(n)\times X^n \to Y.$ If the space $X$ is a topological monoid then
we may equivalently describe $f:X \to Y$ by simply replacing the action of $\cal{J}(n)$ with an action of
 $\cal{CK}(n)$ in the
definition of $A_n$ map.
\end{remark}

\section{Vertex Combinatorics}\label{four}
 Now we mention results regarding the counting of
the vertices of the composihedra. Recall that vertices correspond bijectively to domain equivalence classes of
painted binary trees.

\begin{theorem}
 The number of vertices $a_n$ of the $n^{th}$ composihedron is given recursively by:
$$
a_n = 1 +  \sum_{i=1}^{n-1}a_ia_{n-i}
$$
where $a_0 = 0.$
\end{theorem}
\begin{proof}By a weighted tree we refer to a tree with a real number assigned to each of its leaves. The total weight of
the tree is the sum of the weights of its leaves.
The set of domain equivalence classes of painted binary trees with $n$-leaves is in bijection with the
set of  binary weighted trees of total weight $n$ where leaves have positive integer weights.
This  is evident from the representation of a
domain equivalence class by its least refined member.
 Here is a picture that illustrates the general case of the bijection.
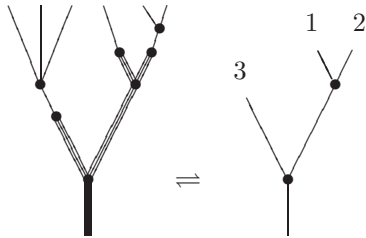
\begin{footnotesize}
 $$
\xy  0;/r.25pc/:
  (-10,22)*=0{}="a"; (-2,22)*=0{}="b"; (-6,22)*=0{}="b2";
  (2,22)*=0{}="c"; (10,22)*=0{}="d"; (7,22)*=0{}="c2"; (9,19)*=0{\bullet}="c3";
  (-6,12)*=0{\bullet}="v1"; (6,12)*=0{\bullet}="v2";
  (0,0)*=0{\bullet}="v3";
  (0,-7)*=0{}="v4";
  (4,16)*=0{\bullet}="va";
  (8,16)*=0{\bullet}="vb";
  (-4,8)*=0{\bullet}="vc" \ar@{};
 "a" ;"v1" \ar@{-};"a" ;"v1" \ar@{-};
 "v1" ;"vc" \ar@{-};
 "b" ;"v1" \ar@{-}  ;
 "b2" ;"v1" \ar@{-}  ;
 "c2" ;"c3" \ar@{-};
 "c" ;"va" \ar@{-} ;
 "d" ;"vb" \ar@{-} \ar@{};
 "va";"v2" \ar@{=} ;"va";"v2" \ar@{=} ;
 "vb";"v2" \ar@{=} ;
 "v2" ;"v3" \ar@{=};
 "vc" ;"v3" \ar@{=} ;
 "v3" ;"v4" \ar@{=} \ar@{};
 %
 %
 "va";"v2" \ar@{-} ;"va";"v2" \ar@{-} ;
 "vb";"v3" \ar@{-} ;
 "vc" ;"v3" \ar@{-} ;
 "v3" ;"v4" \ar@{-} ;
 \endxy
\rightleftharpoons
 \xy  0;/r.25pc/:
  (-10,20)*=0{}="a"; (-2,20)*=0{}="b"; (-6,20)*=0{}="b2";
  (2,20)*=0{}="c"; (10,20)*=0{}="d"; (6,20)*=0{}="c2";
  (-6,12)*{\txt{3 \\\\}}="v1"; (6,12)*{\bullet}="v2";
  (0,0)*=0{\bullet}="v3";
  (0,-7)*=0{}="v4";
  (3,18)*{\txt{1 \\\\}}="va";
  (9,18)*{\txt{2 \\\\}}="vb";
  (-4,8)*=0{}="vc" \ar@{-};
  "v1" ;"vc" \ar@{-};
  "va";"v2" \ar@{-} ;"va";"v2" \ar@{-} ;
 "vb";"v3" \ar@{-} ;
 "vc" ;"v3" \ar@{-} ;
 "v3" ;"v4" \ar@{-} ;
 \endxy
$$
\end{footnotesize}
 There is one weighted 1-leaved tree
with total weight $n$. Now we count the binary trees
with total weight $n$  that have at least one trivalent node. Each of these
consists of a choice of two binary subtrees whose root is the initial trivalent node, and
whose weights must sum to $n.$ Thus we sum over the ways that $n$ can be split into two natural
numbers.
\end{proof}
\begin{remark}
This formula gives the sequence which begins: $$0,1, 2, 5, 15, 51, 188, 731, 2950, 12235\dots.$$ It is sequence
A007317 of the On-line Encyclopedia of integer sequences.
This formula for the sequence was originally stated by Benoit Cloitre.
\end{remark}

The recursive formula  above yields the
equation
$$A(x) = \frac{x}{1-x} + (A(x))^2$$
where $A(x)$ is the ordinary generating function of the sequence $a_n$ above. Thus by use
of the quadratic formula we have
 $$A(x) = \frac{x}{1-x}c(\frac{x}{1-x}).$$
where $c(x) = \frac{1-\sqrt{1-4x}}{2x}$ is the generating function of the Catalan numbers.
This formula was originally derived by Emeric Deutsch from a comment of Michael Somos.

Now if we want to generate the sequence $\{a_{n+1}\}_{n=0}^{\infty} = 1,2,5,15,51,\dots,$ then we must
divide our generating function by $x$, to get $\frac{1}{1-x}c(\frac{x}{1-x}).$ This we recognize as the
generating function of the binary transform of the the Catalan numbers, from the definition of binary transform
in \cite{Barry}. Therefore a direct formula for the sequence is given by
$$a_{n+1} = \sum_{k=0}^n {n\choose k}C(k)$$ where $C(n)$ are the Catalan numbers.

\begin{remark}
This sequence also counts the non-commutative non-associative partitions of
    $n$.
     Two other combinatorial items that this sequence enumerates are: the number
    of Schroeder paths (i.e. consisting of steps
                $U=(1,1),D=(1,-1),H=(2,0)$ and never going below the x-axis) from
               $(0,0)$ to $(2n-2,0)$, with no peaks at even level; and the
              number of tree-like polyhexes (including the non-planar \emph{helicenic} polyhexes) with $n$ or fewer
               cells \cite{DS}.
\end{remark}

\section{Convex hull realizations}\label{five}

In \cite{loday} Loday gives an algorithm for taking the binary trees with $n$ leaves and finding
for each an extremal point in {\bf R}$^{n-1}$; together whose convex hull is ${\cal K}(n),$ the
$(n-2)$-dimensional associahedron. Note that Loday writes formulas with the convention that the
number of leaves is $n+1,$ where we instead always use $n$ to refer to the number of leaves. Given
a (non-painted) binary $n$-leaved tree $t,$ Loday arrives at a point $M(t)$ in {\bf R}$^{n-1}$ by
calculating a coordinate from each trivalent node. These are ordered left to right based upon the
ordering of the leaves from left to right. Following Loday we number the leaves $0,1,\dots,n-1 $
and the nodes $1,2,\dots,n-1.$ The $i^{th}$ node is ``between'' leaf $i-1$ and leaf $i$ where
``between'' might be described to mean that a rain drop falling between those leaves would be
caught at that node. Each trivalent node has a left and right branch, which each support a subtree.
To find the Loday coordinate for the  $i^{th}$ node we take the product of the number of leaves of
the left subtree ($l_i$) and the number of leaves of the right subtree ($r_i$) for that node. Thus
$M(t) = (x_1, \dots x_{n-1})$ where $x_i = l_ir_i$. Loday proves that the convex hull of the points
thus calculated for all $n$-leaved binary trees  is the $n^{th}$ associahedron. He also shows that
the points thus calculated all lie in the $n-2$ dimensional affine hyperplane $H$ given by the
equation $x_1+\dots+x_{n-1} = S(n-1) = {1\over 2}n(n-1).$

In \cite{multi} we adjust Loday's algorithm to apply to painted binary trees as described above, with only nodes of
type (1), (2), and (3), by choosing a number $q \in (0,1).$ Then given a painted binary tree $t$
with $n$ leaves we calculate a point $M_q(t)$ in {\bf R}$^{n-1}$ as follows: we begin by finding
the coordinate for each trivalent node from left to right given by Loday's algorithm, but if the
node is of type (1) (unpainted, or colored by the domain) then its new coordinate is found by
further multiplying its Loday coordinate by $q$.  Thus
$$M_q(t) = (x_1, \dots x_{n-1}) \text{ where } x_i = \begin{cases}
    ql_ir_i,& \text{if node $i$ is type (1)} \\
    l_ir_i,& \text{if node $i$ is type (2).}
\end{cases} $$
Note that whenever we speak of the numbered nodes ($1,\dots, n-1$  from left to right) of a binary
tree, we are referring only to the trivalent nodes, of type (1) or (2). For an example, let us
calculate the point in {\bf R}$^3$ which corresponds to the 4-leaved tree:

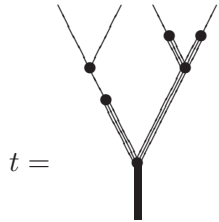
\begin{small}
 $$ t =
\xy  0;/r.25pc/:
  (-10,20)*=0{}="a"; (-2,20)*=0{}="b";
  (2,20)*=0{}="c"; (10,20)*=0{}="d";
  (-6,12)*=0{\bullet}="v1"; (6,12)*=0{\bullet}="v2";
  (0,0)*=0{\bullet}="v3";
  (0,-7)*=0{}="v4";
  (4,16)*=0{\bullet}="va";
  (8,16)*=0{\bullet}="vb";
  (-4,8)*=0{\bullet}="vc" \ar@{};
 "a" ;"v1" \ar@{-};"a" ;"v1" \ar@{-};
 "v1" ;"vc" \ar@{-};
 "b" ;"v1" \ar@{-}  ;
 "c" ;"va" \ar@{-} ;
 "d" ;"vb" \ar@{-} \ar@{};
 "va";"v2" \ar@{=} ;"va";"v2" \ar@{=} ;
 "vb";"v2" \ar@{=} ;
 "v2" ;"v3" \ar@{=};
 "vc" ;"v3" \ar@{=} ;
 "v3" ;"v4" \ar@{=} \ar@{};
 %
 %
 "va";"v2" \ar@{-} ;"va";"v2" \ar@{-} ;
 "vb";"v3" \ar@{-} ;
 "vc" ;"v3" \ar@{-} ;
 "v3" ;"v4" \ar@{-} ;
 \endxy
$$
\end{small}

Then $M_q(t) = (q, 4, 1). $

Now we turn to consider the case when in the  algorithm  described above we let $q=0.$
$$M_0(t) = (x_1, \dots x_{n-1}) \text{ where } x_i = \begin{cases}
    0,& \text{if node $i$ is type (1)} \\
    l_ir_i,& \text{if node $i$ is type (2).}
\end{cases} $$

For the same tree $t$ in the above example we have $M_0(t) = (0, 4, 1). $
\begin{lemma}\label{zero}
 If (painted binary) tree $t$ is domain equivalent to $t'$ then $M_0(t) = M_0(t').$
That
is, each domain equivalence class of binary painted trees contributes exactly one point.
\end{lemma}
\begin{proof}
This is clear from the fact that all the unpainted nodes of a tree contribute a 0 coordinate.
\end{proof}
\begin{theorem}\label{main} The convex hull of all the resulting points $M_0(t)$ for $t$ in the set of $n$-leaved binary
painted trees is the $n^{th}$ composihedron. That is, our convex hull is combinatorially
equivalent to the CW-complex $\cal{CK}(n)$.
\end{theorem}
The proof will follow in section~\ref{proofsec}.
$$$$
Here are all the painted binary trees with 3 leaves, together with their points $M_0(t)\in${\bf
R}$^2.$
\begin{small}
\begin{tabular}{llllll}
&$ M_0\left(  \xy  0;/r.15pc/:
  (-10,14)*{\txt{\\\\}}="aa";
  (-2,14)*{}="ab";
  (2,14)*{}="ac";
  (10,14)*{}="ad";
  (-6,6)*=0{\bullet}="av1";
  (6,6)*=0{}="av2";
  (0,-6)*=0{\bullet}="av3";
  (0,-13)*=0{}="av4";
  (4,10)*=0{}="ava";
  (8,10)*=0{}="avb";
  (-4,10)*=0{\bullet}="ave";
  (-8,10)*=0{\bullet}="avf";
  (-4,2)*=0{}="avc" ;
  (4,2)*=0{\bullet}="avd" \ar@{};
 "aa" ;"avc" \ar@{-};"aa" ;"avc" \ar@{-};
 "ab" ;"av1" \ar@{-}  ;
  "ad" ;"avb" \ar@{-} \ar@{};
 "ave";"av1" \ar@{=} ;"ave";"av1" \ar@{=} ;
 "avf";"avc" \ar@{=} ;
 "avc" ;"av3" \ar@{=} ;
 "avd" ;"av3" \ar@{=} ;
 "av3" ;"av4" \ar@{=} \ar@{};
"avb";"av3" \ar@{-} ;
 "avb";"av3" \ar@{-} ;
 "avc" ;"av3" \ar@{-} ;
 "av3" ;"av4" \ar@{-} ;
 \endxy
\right) = (1,2),$ &  $M_0\left( \xy  0;/r.15pc/:
  (-10,14)*{\txt{\\\\}}="ba";
  (-2,14)*=0{}="bb";
  (2,14)*=0{}="bc";
  (10,14)*=0{}="bd";
  (-6,6)*=0{}="bv1";
  (6,6)*=0{\bullet}="bv2";
  (0,-6)*=0{\bullet}="bv3";
  (0,-13)*{}="bv4";
  (4,10)*=0{\bullet}="bva";
  (8,10)*=0{\bullet}="bvb";
  (-4,10)*=0{}="bve";
  (-8,10)*=0{}="bvf";
  (-4,2)*=0{\bullet}="bvc" ;
  (4,2)*=0{}="bvd" \ar@{};
 "ba" ;"bvc" \ar@{-};"ba" ;"bvc" \ar@{-};
  "bc" ;"bva" \ar@{-} ;
 "bd" ;"bvb" \ar@{-} \ar@{};
 "bva";"bv2" \ar@{=} ;"bva";"bv2" \ar@{=} ;
 "bvb";"bv3" \ar@{=} ;
 "bvc" ;"bv3" \ar@{=} ;
 "bv3" ;"bv4" \ar@{=} \ar@{};
 "bva";"bv2" \ar@{-} ;"bva";"bv2" \ar@{-} ;
 "bvb";"bv3" \ar@{-} ;
 "bvc" ;"bv3" \ar@{-} ;
 "bv3" ;"bv4" \ar@{-} ;
 \endxy\right) = (2,1)$ \\
   $M_0\left( \xy  0;/r.15pc/:
  (-10,14)*{\txt{\\\\}}="ca";
  (-2,14)*=0{}="cb";
  (2,14)*=0{}="cc";
  (10,14)*=0{}="cd";
  (-6,6)*=0{\bullet}="cv1";
  (6,6)*=0{}="cv2";
  (0,-6)*=0{\bullet}="cv3";
  (0,-13)*{}="cv4";
  (4,10)*=0{}="cva";
  (8,10)*=0{}="cvb";
  (-4,10)*=0{}="cve";
  (-8,10)*=0{}="cvf";
  (-4,2)*=0{\bullet}="cvc" ;
  (4,2)*=0{\bullet}="cvd" \ar@{};
 "ca" ;"cvc" \ar@{-};"ca" ;"cvc" \ar@{-};
 "cb" ;"cv1" \ar@{-}  ;
  "cd" ;"cvb" \ar@{-} \ar@{};
 "cvc" ;"cv3" \ar@{=} ;
 "cvc" ;"cv3" \ar@{=} ;
 "cvd" ;"cv3" \ar@{=} ;
 "cv3" ;"cv4" \ar@{=} \ar@{};
"cvb";"cv3" \ar@{-} ;
 "cvb";"cv3" \ar@{-} ;
 "cvc" ;"cv3" \ar@{-} ;
 "cv3" ;"cv4" \ar@{-} ;
 \endxy\right) = (0,2)$ &&&   $ M_0\left( \xy  0;/r.15pc/:
  (-10,14)*{\txt{\\\\}}="da";
  (-2,14)*=0{}="db";
  (2,14)*=0{}="dc";
  (10,14)*=0{}="dd";
  (-6,6)*=0{}="dv1";
  (6,6)*=0{\bullet}="dv2";
  (0,-6)*=0{\bullet}="dv3";
  (0,-13)*{}="dv4";
  (4,10)*=0{}="dva";
  (8,10)*=0{}="dvb";
  (-4,10)*=0{}="dve";
  (-8,10)*=0{}="dvf";
  (-4,2)*=0{\bullet}="dvc" ;
  (4,2)*=0{\bullet}="dvd" \ar@{};
 "da" ;"dvc" \ar@{-};"da" ;"dvc" \ar@{-};
  "dc" ;"dva" \ar@{-} ;
 "dd" ;"dvb" \ar@{-} \ar@{};
"dvd" ;"dv3" \ar@{=} ;"dvd" ;"dv3" \ar@{=} ;
 "dvc" ;"dv3" \ar@{=} ;
 "dv3" ;"dv4" \ar@{=} \ar@{};
 "dva";"dv2" \ar@{-} ;"dva";"dv2" \ar@{-} ;
 "dvb";"dv3" \ar@{-} ;
 "dvc" ;"dv3" \ar@{-} ;
 "dv3" ;"dv4" \ar@{-} ;
 \endxy\right) = (2,0)$ \\
 &   $ M_0\left(\xy  0;/r.15pc/:
  (-10,14)*{\txt{\\\\}}="ea";
  (-2,14)*=0{}="eb";
  (2,14)*=0{}="ec";
  (10,14)*=0{}="ed";
  (-6,6)*=0{\bullet}="ev1";
  (6,6)*=0{}="ev2";
  (0,-6)*=0{\bullet}="ev3";
  (0,-14)*{}="ev4";
  (0,-10)*{\bullet}="ev5";
  (4,10)*=0{}="eva";
  (8,10)*=0{}="evb";
  (-4,10)*=0{}="eve";
  (-8,10)*=0{}="evf";
  (-4,2)*=0{}="evc" ;
  (4,2)*=0{}="evd" \ar@{};
 "ea" ;"evc" \ar@{-};"ea" ;"evc" \ar@{-};
 "eb" ;"ev1" \ar@{-}  ;
  "ed" ;"evb" \ar@{-} \ar@{};
   "ev5" ;"ev4" \ar@{=}; "ev5" ;"ev4" \ar@{=} \ar@{};
"evb";"ev3" \ar@{-} ;
 "evb";"ev3" \ar@{-} ;
 "evc" ;"ev3" \ar@{-} ;
 "ev3" ;"ev4" \ar@{-} ;
 \endxy \right) = (0,0),$ &  $      M_0\left( \xy  0;/r.15pc/:
  (-10,14)*{\txt{\\\\}}="fa";
  (-2,14)*=0{}="fb";
  (2,14)*=0{}="fc";
  (10,14)*=0{}="fd";
  (-6,6)*=0{}="fv1";
  (6,6)*=0{\bullet}="fv2";
  (0,-6)*=0{\bullet}="fv3";
  (0,-14)*{}="fv4";
  (0,-10)*{\bullet}="fv5";
  (4,10)*=0{}="fva";
  (8,10)*=0{}="fvb";
  (-4,10)*=0{}="fve";
  (-8,10)*=0{}="fvf";
  (-4,2)*=0{}="fvc" ;
  (4,2)*=0{}="fvd" \ar@{};
 "fa" ;"fvc" \ar@{-};"fa" ;"fvc" \ar@{-};
  "fc" ;"fva" \ar@{-} ;
 "fd" ;"fvb" \ar@{-} \ar@{};
  "fv5" ;"fv4" \ar@{=}; "fv5" ;"fv4" \ar@{=} \ar@{};
 "fva";"fv2" \ar@{-} ;"fva";"fv2" \ar@{-} ;
 "fvb";"fv3" \ar@{-} ;
 "fvc" ;"fv3" \ar@{-} ;
 "fv3" ;"fv4" \ar@{-} ;
 \endxy\right) = (0,0)$ \\
\end{tabular}
\end{small}

$$$$
Note that the bottom two points are both the origin.
 The convex
hull of the five total distinct points appears as follows:
\begin{small}
$$
\xy 0;/r3pc/: (0,0);(0,3) \ar@{..>} ; (0,0);(0,3) \ar@{..>} ; (0,0); (3,0) \ar@{-} ;
  (1,2);(2,1) \ar@{-};(1,2)*=0{\bullet};(2,1)*=0{\bullet} \ar@{-};
  (1,2);(0,2)*=0{\bullet}\ar@{-};
  (0,2);(0,0)*=0{\bullet} \ar@{-};
  (0,0);(0,0)*=0{\bullet} \ar@{-};
  (0,0);(2,0)*=0{\bullet} \ar@{-};
  (2,0);(2,1)*=0{\bullet} \ar@{-};
  \endxy
$$
\end{small}
The (redundant) list of vertices for $\cal{CK}(4)$ based on painted binary trees with 4 leaves is:
\newline

\begin{tabular}{lllll}
 (1, 2 ,3)&
 (0 ,2 ,3)&
 (0 ,0 ,3)&
 (0, 0 ,0)\\
 (2, 1, 3)&
 (2 ,0 ,3)&
 (0 ,0 ,3)&
 (0, 0 ,0)\\
 (3 ,1 ,2)&
 (3, 0, 2)&
 (3 ,0 ,0)&
 (0, 0 ,0)\\
 (3, 2, 1)&
 (3 ,2, 0)&
 (3 ,0, 0)&
 (0 ,0 ,0)\\
 (1 ,4 ,1)&
 (0, 4, 1)&
 (1, 4, 0)&
 (0, 4, 0)&
 (0, 0 ,0)\\
\end{tabular}
\newline
These are suggestively listed as a table where the first column is made up of the coordinates
calculated by Loday for $\cal{K}(4)$, which here correspond to trees with every trivalent node
entirely painted. The rows may be found by applying the factor $0$ to each coordinate in turn, in
order of increasing size of those coordinates. Here is the convex hull of the fifteen total distinct points, where we
see that each row of the table corresponds to shortest paths from the big pentagon to the origin.
 Of course sometimes there are multiple such paths.
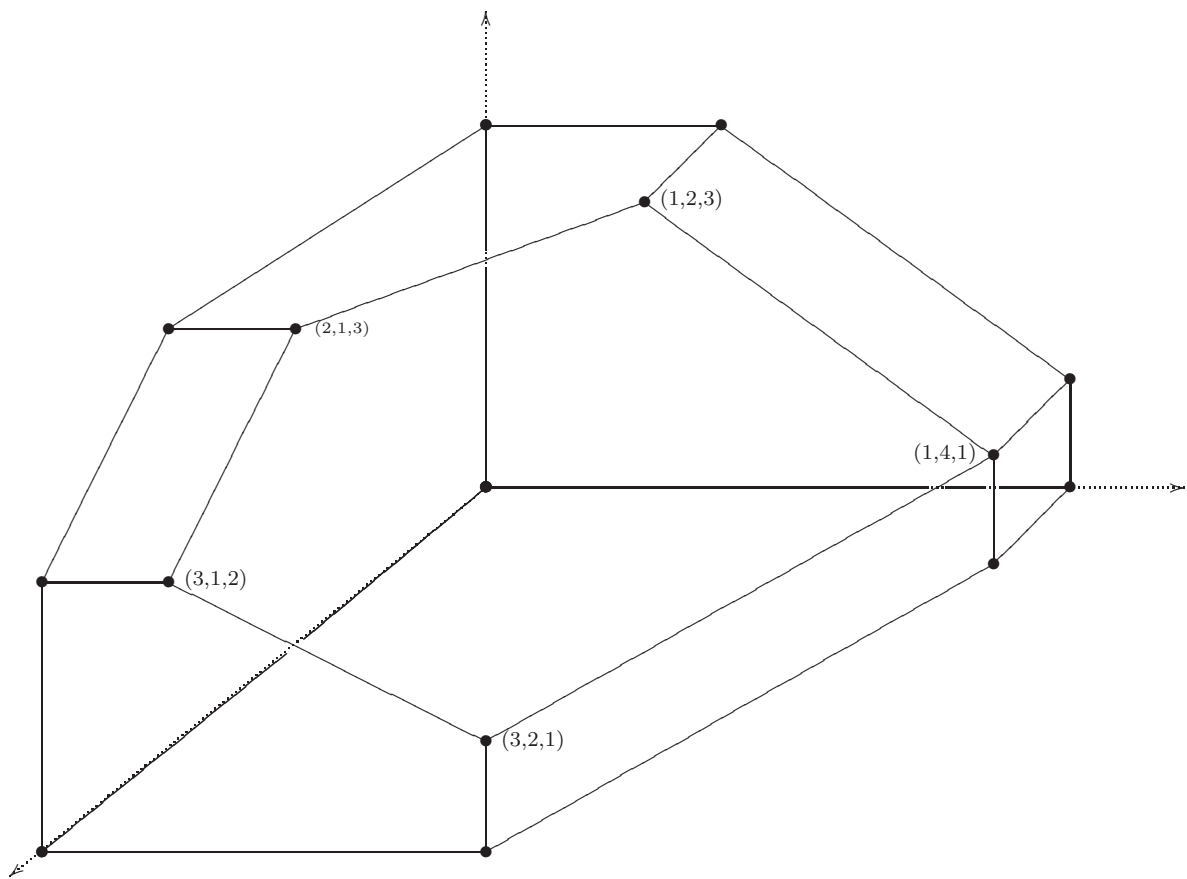
\newline

\begin{small}
$$
\xy 0;/r1pc/:
(0,1);(0,16) \ar@{.>} ; (0,1);(0,16) \ar@{.>} ; (0,1);(-15,-11.25)\ar@{.>}; (0,1); (22,1) \ar@{-} ;
  (0,1)*=0{\bullet}="1a";
  (0,1)*=0{\bullet}="1b";
  (0,1)*=0{\bullet}="1c";
  (0,1)*=0{\bullet}="1d";
  (0,1)*=0{\bullet}="1e";
(-10,-2)*=0{\bullet}="2a";
  (-10,-2)*=0{^{(3,1,2)}\hspace{-.5in}};
  (0,-7)*=0{\bullet}="2b";
  (0,-7)*=0{_{(3,2,1)}\hspace{-.5in}};
  (16,2)*=0{\bullet}="2c";
  (16,2)*=0{^{(1,4,1)}\hspace{.5in}};
  (5,10)*=0{\bullet}="2d";
  (5,10)*=0{^{(1,2,3)}\hspace{-.5in}};
  (-6,6)*=0{\bullet}="2e";
  (-6,6)*=0{_{_{(2,1,3)}}\hspace{-.5in}};
  (-14,-2)*=0{\bullet}="3a";
  (-14,-10.5)*=0{\bullet}="3b";
  (-14,-10.5)*=0{\bullet}="3c";
  (0,-10.5)*=0{\bullet}="3d";
  (16,-1.418)*=0{\bullet}="3e";
  (18.418,1)*=0{\bullet}="3f";
  (18.418,4.418)*=0{\bullet}="3g";
  (7.418,12.418)*=0{\bullet}="3h";
  (0,12.418)*=0{\bullet}="3i";
  (0,12.418)*=0{\bullet}="3j";
  (-10,6)*=0{\bullet}="3k";
 (14.222,1)*{\hole}="x"; (14.222,1)*{\hole}="x";
(16,1)*{\hole}="y";
  (0,8.2)*{\hole}="z";
 (0,8.2)*{\hole}="w";
 (-6,-4)*{\hole}="p";
  (-6,-4)*{\hole}="q";
  "1a";"p" \ar@{-}; "1a";"p"\ar@{-};
  "p";"3b"  \ar@{-};
    "1b";"q" \ar@{-};
  "q";"3c" \ar@{-};
 "1c";"x" \ar@{-};
 "x";"y" \ar@{-};
 "y";"3f" \ar@{-};
 "1d";"w" \ar@{-};
 "w";"3i" \ar@{-};
 "1e";"z" \ar@{-};
 "z";"3j" \ar@{-};
 "2a";"3a" \ar@{-};
 "2b";"3d" \ar@{-};
  "2c";"3e" \ar@{-};
  "2c";"3g" \ar@{-};
 "2d";"3h" \ar@{-};
 "2e";"3k" \ar@{-};
"1a";"1b" \ar@{-}; "1a";"1b" \ar@{-};
 "1b";"1c" \ar@{-};
 "1c";"1d" \ar@{-};
 "1d";"1e" \ar@{-};
 "1e";"1a" \ar@{-};
 "2a";"2b" \ar@{-};
 "2b";"2c" \ar@{-};
 "2c";"2d" \ar@{-};
 "2d";"2e" \ar@{-};
 "2e";"2a" \ar@{-};
 "3a";"3b" \ar@{-};
  "3b";"3c" \ar@{-};
  "3c";"3d" \ar@{-};
  "3d";"3e" \ar@{-};
  "3e";"3f" \ar@{-};
  "3f";"3g" \ar@{-};
  "3g";"3h" \ar@{-};
  "3h";"3i" \ar@{-};
  "3i";"3j" \ar@{-};
  "3j";"3k" \ar@{-};
  "3k";"3a" \ar@{-};
     \endxy
$$
\end{small}
\newline

To see the picture of $\cal{CK}(4)$ that is in Figure 1 of this paper, just rotate this view of the convex hull
by 90 degrees clockwise.
To compare to other pictures of $\cal{CK}(4)$ in this paper note the
 smallest quadrilateral facet in this picture containing the unique vertex with four incident edges.

To see a rotatable version of the convex hull which is the fourth composihedron, enter the
following homogeneous coordinates into the Web Demo of polymake (with option visual), at
\url{http://www.math.tu-berlin.de/polymake/index.html#apps/polytope}. Indeed polymake was instrumental in
the experimental phase of this research \cite{poly}.

\begin{align*}
POINTS\\
 1 ~1~ 2~ 3\\
 1~ 0~ 2 ~3\\
 1 ~0~ 0~ 3\\
 1 ~0 ~0~ 0\\
 1~ 2~ 1~ 3\\
 1 ~2~ 0 ~3\\
 1~ 3~ 1~ 2\\
 1~ 3~ 0~ 2\\
 1~ 3~ 0 ~0\\
 1~ 3~ 2~ 1\\
 1~ 3~ 2~ 0\\
 1~ 1~ 4~ 1\\
 1~ 0~ 4 ~1\\
 1~ 1~ 4~ 0\\
 1~ 0 ~4 ~0\\
\end{align*}

\begin{remark}
It is also fairly simple to devise a mapping from $n$-leaved painted binary trees to Euclidean space which
reflects the quotient of the multiplihedron by range equivalence, where $f(a)(f(b)f(c)) = (f(a)f(b))f(c).$
It may be done by reflecting such equivalences as:
$$ 
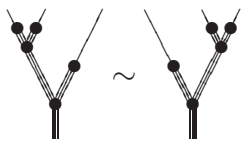
\xy  0;/r.15pc/:
(-10,14)*{\txt{\\\\}}="aa";
  (-2,14)*{}="ab";
  (2,14)*{}="ac";
  (10,14)*{}="ad";
  (-6,6)*=0{\bullet}="av1";
  (6,6)*=0{}="av2";
  (0,-6)*=0{\bullet}="av3";
  (0,-13)*=0{}="av4";
  (4,10)*=0{}="ava";
  (8,10)*=0{}="avb";
  (-4,10)*=0{\bullet}="ave";
  (-8,10)*=0{\bullet}="avf";
  (-4,2)*=0{}="avc" ;
  (4,2)*=0{\bullet}="avd" \ar@{};
 "aa" ;"avc" \ar@{-};"aa" ;"avc" \ar@{-};
 "ab" ;"av1" \ar@{-}  ;
  "ad" ;"avb" \ar@{-} \ar@{};
 "ave";"av1" \ar@{=} ;"ave";"av1" \ar@{=} ;
 "avf";"avc" \ar@{=} ;
 "avc" ;"av3" \ar@{=} ;
 "avd" ;"av3" \ar@{=} ;
 "av3" ;"av4" \ar@{=} \ar@{};
"avb";"av3" \ar@{-} ;
 "avb";"av3" \ar@{-} ;
 "avc" ;"av3" \ar@{-} ;
 "av3" ;"av4" \ar@{-} ;
 \endxy
\sim \xy  0;/r.15pc/:
  (-10,14)*{\txt{\\\\}}="ba";
  (-2,14)*=0{}="bb";
  (2,14)*=0{}="bc";
  (10,14)*=0{}="bd";
  (-6,6)*=0{}="bv1";
  (6,6)*=0{\bullet}="bv2";
  (0,-6)*=0{\bullet}="bv3";
  (0,-13)*{}="bv4";
  (4,10)*=0{\bullet}="bva";
  (8,10)*=0{\bullet}="bvb";
  (-4,10)*=0{}="bve";
  (-8,10)*=0{}="bvf";
  (-4,2)*=0{\bullet}="bvc" ;
  (4,2)*=0{}="bvd" \ar@{};
 "ba" ;"bvc" \ar@{-};"ba" ;"bvc" \ar@{-};
  "bc" ;"bva" \ar@{-} ;
 "bd" ;"bvb" \ar@{-} \ar@{};
 "bva";"bv2" \ar@{=} ;"bva";"bv2" \ar@{=} ;
 "bvb";"bv3" \ar@{=} ;
 "bvc" ;"bv3" \ar@{=} ;
 "bv3" ;"bv4" \ar@{=} \ar@{};
 "bva";"bv2" \ar@{-} ;"bva";"bv2" \ar@{-} ;
 "bvb";"bv3" \ar@{-} ;
 "bvc" ;"bv3" \ar@{-} ;
 "bv3" ;"bv4" \ar@{-} ;
 \endxy
$$
$\dots$ by mapping them both to a single vertex in $\mathbb{R}^2.$ One possible map is
$$M'(t) = (x_1, \dots x_{n-1}) \text{ where } x_i = \begin{cases}
    ql_ir_i,& \text{if node $i$ is type (1)} \\
    i(n-i),& \text{if node $i$ is type (2).}
\end{cases} $$

This will yield
a new realization of $\cal{K}(n+1)$ in $\mathbb{R}^{(n-1)},$ with vertices corresponding to range equivalence
classes of $n$-leaved painted trees. Here is
the (redundant) list of vertices for $\cal{K}(5)$ based on painted binary trees with 4 leaves:
The list of vertices for ${\cal J}(4)$ based on painted binary trees with 4 leaves, for $q= {1
\over 2},$ is:
\newline

\begin{tabular}{lllll}
 (3, 4 ,3)&
 (1/2 ,4 ,3)&
 (1/2 ,2/2 ,3)&
 (1/2, 2/2 ,3/2)\\
 (3, 4, 3)&
 (3 ,1/2 ,3)&
 (2/2 ,1/2 ,3)&
 (2/2, 1/2 ,3/2)\\
 (3 ,4 ,3)&
 (3, 1/2, 3)&
 (3 ,1/2 ,2/2)&
 (3/2, 1/2 ,2/2)\\
 (3, 4, 3)&
 (3 ,4, 1/2)&
 (3 ,2/2, 1/2)&
 (3/2 ,2/2 ,1/2)\\
 (3 ,4 ,3)&
 (1/2, 4, 3)&
 (3, 4, 1/2)&
 (1/2, 4, 1/2)&
 (1/2, 4/2 ,1/2)\\
\end{tabular}
\newline
These are suggestively listed as a table where the first column is the single image of
all the trees with every trivalent node
entirely painted. Indeed there are 14 total different vertices. Here is the convex hull of those vertices.
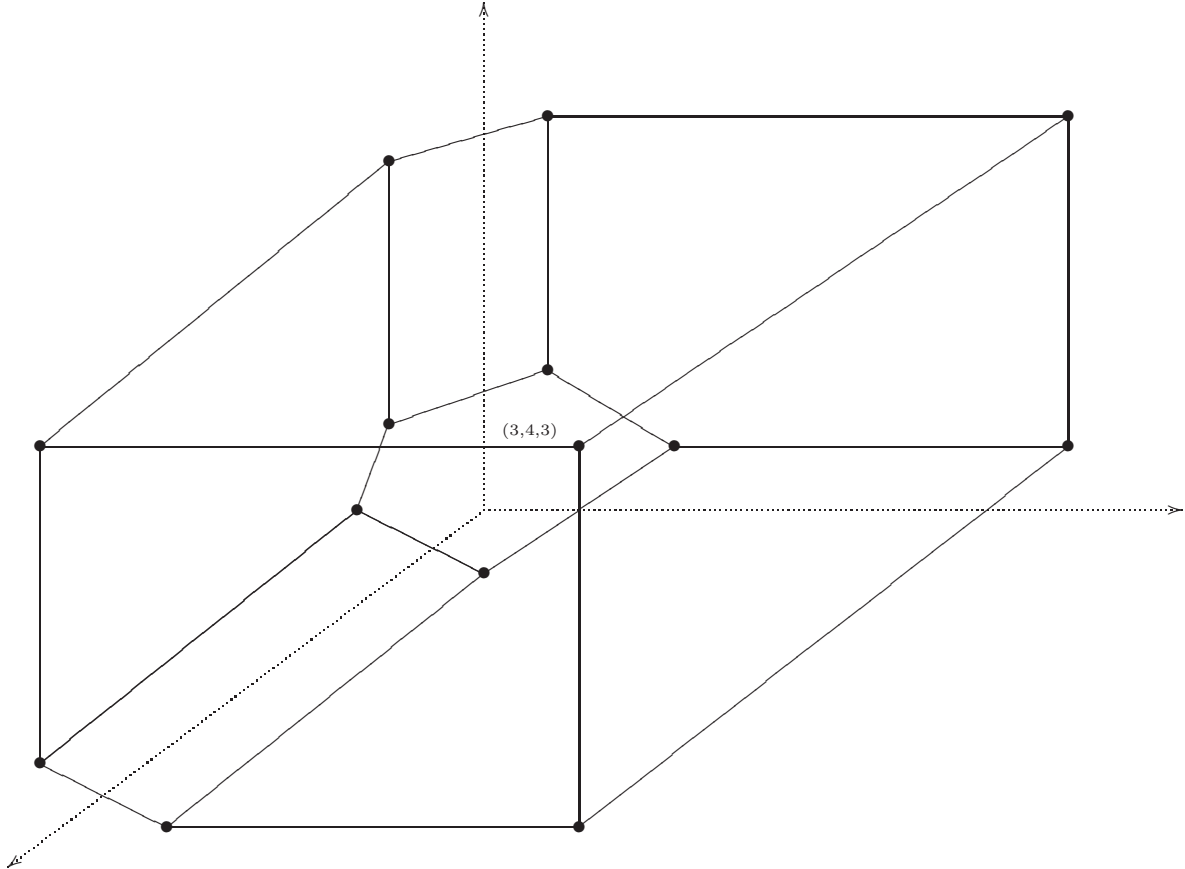
\begin{small}
$$
\xy 0;/r1pc/:
   (0,0);(0,16) \ar@{.>} ; (0,0);(0,16) \ar@{.>} ; (0,0);(-15,-11.25)\ar@{.>}; (0,0); (22,0) \ar@{-} ;
      (-4,0)*=0{\bullet}="1a";
  (0,-2)*=0{\bullet}="1b";
  (6,2)*=0{\bullet}="1c";
  (2,4.418)*=0{\bullet}="1d";
  (-3,2.7)*=0{\bullet}="1e";
%
  %
  (-14,-8)*=0{\bullet}="3b";
  (-10,-10)*=0{\bullet}="3c";
  (18.418,2)*=0{\bullet}="3f";
  (2,12.418)*=0{\bullet}="3i";
  (-3,11)*=0{\bullet}="3j";
  (-14,2)*=0{\bullet}="4a";
  (3,-10)*=0{\bullet}="4b";
  (18.418,12.418)*=0{\bullet}="4c";
  (3,2)*=0{\bullet}="4d";
  (3,2.5)*=0{_{_{(3,4,3)}}\hspace{.5in}};
%
  "1a";"3b"  \ar@{-};"1a";"3b"  \ar@{-};
    "1b";"3c" \ar@{-};
 "1c";"3f" \ar@{-};
 "1d";"3i" \ar@{-};
 "1e";"3j" \ar@{-};
 %
 %
"1a";"1b" \ar@{-}; "1a";"1b" \ar@{-};
 "1b";"1c" \ar@{-};
 "1c";"1d" \ar@{-};
 "1d";"1e" \ar@{-};
 "1e";"1a" \ar@{-};
 %
 %
 "4a";"3b" \ar@{-};
  "3b";"3c" \ar@{-};
  "3c";"4b" \ar@{-};
  "4b";"3f" \ar@{-};
  "3f";"4c" \ar@{-};
  "4c";"3i" \ar@{-};
  "3i";"3j" \ar@{-};
  "3j";"4a" \ar@{-};
"4a";"4d" \ar@{-};
"4b";"4d" \ar@{-};
"4c";"4d" \ar@{-};
     \endxy
$$
\end{small}
\end{remark}

\section{Proof of Theorem~\ref{main}}\label{proofsec}
Our algorithm for generating extremal points whose convex hull gives the composihedra can be described as the
 $q\to 0$ limit of the algorithm for the multiplihedra, which was given in \cite{multi}.
 What we need to demonstrate is that
 the limiting process does indeed deliver the
 combinatorial equivalent of the CW-complex $\cal{CK}(n).$

To demonstrate that our convex hulls are each combinatorially equivalent to the
  corresponding  $CW$-complexes of Definition~\ref{recur} we need only check that they both
   have the same vertex-facet incidence. We will show that for
 each $n$
 there is an isomorphism $f$ between the  vertex sets (0-cells) of our convex hull and $\cal{CK}(n)$
 which preserves the sets of vertices
  corresponding to facets; i.e. if $S$ is the set of vertices of a facet of our convex hull then
   $f(S)$ is a vertex set of
  a facet of $\cal{CK}(n).$

  To demonstrate the existence of the isomorphism, noting that the vertices of $\cal{CK}(n)$
  correspond to the domain equivalence classes of binary painted trees, we only
 need to check that the points we calculate from those classes are indeed the vertices of their convex hull.
 Recall that the calculation of a point from a class is well defined by Lemma~\ref{zero}.
 The isomorphism  $f$  is the one that takes the vertex calculated from a certain class to the 0-cell
 associated to the same class. Now a given facet of $\cal{CK}(n)$ corresponds to a tree $T$ which is one of the
 two sorts of trees pictured in Definition~\ref{recur}. To show that our implied isomorphism of
 vertices preserves vertex sets of
facets we need to show that that our facet is the convex hull of the points corresponding to the
classes of binary trees represented by \emph{refinements} of $T$. By refinement of painted trees we refer to the
relationship: $t$ refines $t'$ if $t'$ results from the collapse of some of the internal edges of
$t$.

The proofs of both key points will proceed in tandem, and will be inductive.
That is, we will show that for each facet tree $T$, the points $M_0(t)$ for $t<T$ are precisely those
points that lie in a a bounding hyperplane of our convex hull. Then we will check that  those
points $M_0(t)$ are the extremal points of their convex hull, which is indeed an $(n-2)$-dimensional
polytope of the type decreed in Definition~\ref{recur}.

 We will use the fact that if $P(q)$ and $R(q)$
are two polytopes with some vertices parameterized continuously by $q$ that

$$\lim_{q\to a}(P\times Q) \equiv \lim_{q\to a}P \times \lim_{q\to a}Q$$

\begin{definition}
The lower facets $\cal{CK}_k(n-1,2)$ correspond to minimal lower trees such as:
 
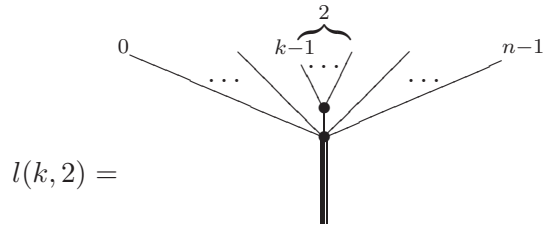
\begin{small}
 $$
\xy (0,0) *{l(k,2) =}\endxy
\xy (0,0) *{ \xy  0;/r.45pc/:
    (-8,22) *{{2 \atop \overbrace{~~~~~~~~~}}};
   (-22,20)*{{}^0}="L";(6,20)*{{}^{n-1}}="R";
  (-14,20)*=0{}="a"; (-10,20)*{{}^{k-1}}="b";
  (-6,20)*=0{}="b2";
  (-2,20)*=0{}="c";
  (-8,14)*=0{\bullet}="v1";
  (-8,8)*=0{}="v3";
  (-8,16)*=0{\bullet}="v0";
  (-15,18) *{\dots};
  (-1,18) *{\dots};
  (-8,19) *{\dots};
   "a" ;"v1" \ar@{-};"a" ;"v1" \ar@{-};
   "L" ;"v1" \ar@{-};
   "R" ;"v1" \ar@{-};
 "b" ; "v0" \ar@{-};
 "b2" ; "v0" \ar@{-};
 "v0"; "v1" \ar@{-};
 "v1" ; "c" \ar@{-};
  "v1" ;"v3" \ar@{-};
  \ar@{} ;
"v1" ;"v3" \ar@{=};"v1" ;"v3" \ar@{=};
 \endxy} \endxy
$$
\end{small}

These are assigned a hyperplane $H_0(l(k,2))$ determined by the equation
$$x_k
= 0.$$

\end{definition}
 Recall that $n-1$ is the number of branches extending from the lowest node. Thus $1\le k \le n-1.$
 Notice that if $q$ were not zero as in definition 5.1 of \cite{multi} then the the equation would appear:
$$x_k = q.$$ We conjecture that these latter hyperplanes would function as replacements for the ones given by $x_k=0$
in that after the replacement the resulting polytope would still be the composihedron.

\begin{definition}
The upper facets $\cal{CK}(t;r_1,\dots ,r_t)$ correspond to upper trees such as:
 
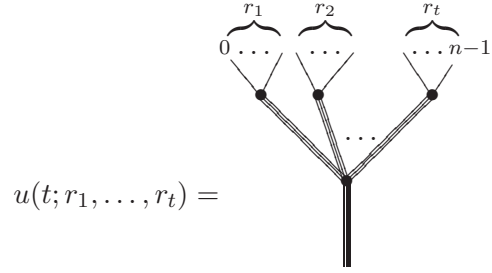
\begin{small}
 $$
\xy (0,0) *{ u(t;r_1,\dots, r_t)= }\endxy
\xy (0,0) *{ \xy  0;/r.45pc/:
  (-.5,23)*{{}^0}="a"; (1.5,25.25) *{{r_1 \atop \overbrace{~~~}}}; (1.5,23) *{\dots}; (3.5,23)*=0{}="b";
(4.5,23)*=0{}="a2"; (6.5,25.25) *{{r_2 \atop \overbrace{~~~}}}; (6.5,23)  *{\dots};
(8.5,23)*=0{}="b2";
 (12,23)*=0{}="a3"; (14,25.25) *{{r_t \atop \overbrace{~~~}}}; (14,23) *{\dots~~}; (16,23)*{~~^{~n-1}}="b3";
   (2,20)*=0{\bullet}="d";
  (6,20)*=0{\bullet}="e"; (14,20)*=0{\bullet}="f";
 (9,17) *{\dots};
  (8,14)*=0{\bullet}="v2";
  (8,8)*=0{}="v4"; \ar@{};
 "d" ;"v2" \ar@{-}; "d" ;"v2" \ar@{-};
"a" ;"d" \ar@{-};
 "b" ;"d" \ar@{-};
  "a2" ;"e" \ar@{-};
   "b2" ;"e" \ar@{-};
 "a3" ;"f" \ar@{-};
 "b3" ;"f" \ar@{-};
 "e" ;"v2" \ar@{-};
 "f" ;"v2" \ar@{-};
  "v2" ;"v4" \ar@{-};
 \ar@{} ;
"d" ;"v2" \ar@{=};"d" ;"v2" \ar@{=};
 "e" ;"v2" \ar@{=};
 "f" ;"v2" \ar@{=};
 "v2" ;"v4" \ar@{=};
 \endxy} \endxy
$$
\end{small}

These are assigned a hyperplane $H_0(u(t;r_1,\dots, r_t))$ determined by the equation
$$x_{r_1} + x_{(r_1+r_2)} +
\dots + x_{(r_1+r_2+\dots +r_{t-1})} = {1 \over 2}\left(n(n-1)-\sum_{i=1}^t r_i(r_i-1)\right)$$ or
equivalently:
 $$x_{r_1} + x_{(r_1+r_2)} + \dots + x_{(r_1+r_2+\dots +r_{t-1})} = \sum_{1\le i<j \le t}r_ir_j.$$
  \end{definition}

 Note that if $t=n$ (so $r_i = 1$ for all $i$) that
this becomes the hyperplane given by
$$x_1 + \dots +
x_{n-1} = {1 \over 2}n(n-1) = S(n-1).$$ Therefore the points $M_0(t)$ for $t$ a binary tree with
only nodes type (2) and (3) will lie in the hyperplane $H$ by Lemma 2.5 of \cite{loday} (using notation $S(n)$ and $H$ as
in that source).
Also note that these upper hyperplanes are exactly the same as those defined for $\cal{J}(n)$ in definition
5.2 of \cite{multi}.

  In order to prove Theorem~\ref{main} it turns out to be expedient to prove a more general result. This
  consists of an even more flexible version of the algorithm for assigning points to binary trees
  in order to achieve a convex hull of those points which is the composihedron. To assign points
  in {\bf R}$^{n-1}$
  to the domain equivalence classes of binary painted trees with $n$ leaves we choose
  an ordered $n$-tuple of positive integers $w_0, \dots, w_{n-1}.$  Now given a tree $t$
we calculate a point $M^{w_0, \dots, w_{n-1}}_0(t)$ in {\bf R}$^{n-1}$ as follows: we begin by
assigning the weight $w_i$ to the $i^{th}$ leaf. We refer to the result as a weighted tree. Then we
modify Loday's algorithm for finding the coordinate for each trivalent node by replacing the number
of leaves of the left and right subtrees with the sums of the weights of the leaves of those
subtrees. Thus we let $L_i = \sum w_k$ where the sum is over the leaves of the subtree supported by
the left branch of the $i^{th}$ node. Similarly we let $R_i = \sum w_k$ where $k$ ranges over the
leaves of the the subtree supported by the right branch.
  Then
$$M^{w_0, \dots, w_{n-1}}_0(t) = (x_1, \dots x_{n-1}) \text{ where } x_i = \begin{cases}
    0,& \text{if node $i$ is type (1)} \\
    L_iR_i,& \text{if node $i$ is type (2).}
\end{cases} $$
Note that the original points $M_0(t)$ are recovered if $w_i = 1$ for  $i=0,\dots,n-1.$ Thus
proving that the convex hull of the points $M^{w_0, \dots, w_{n-1}}_0(t)$ where $t$ ranges over the
binary painted trees with $n$ leaves is the $n^{th}$ composihedron will imply the main theorem.
For an example,  let us calculate the point in {\bf R}$^3$ which corresponds to the 4-leaved tree:

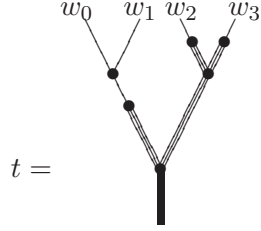
\begin{small}
 $$ t =
\xy  0;/r.25pc/:
  (-10,20)*{ w_0~}="a"; (-2,20)*{w_1~}="b";
  (2,20)*{~w_2}="c"; (10,20)*{~w_3}="d";
  (-6,12)*=0{\bullet}="v1"; (6,12)*=0{\bullet}="v2";
  (0,0)*=0{\bullet}="v3";
  (0,-7)*=0{}="v4";
  (4,16)*=0{\bullet}="va";
  (8,16)*=0{\bullet}="vb";
  (-4,8)*=0{\bullet}="vc" \ar@{};
 "a" ;"v1" \ar@{-};"a" ;"v1" \ar@{-};
 "v1" ;"vc" \ar@{-};
 "b" ;"v1" \ar@{-}  ;
 "c" ;"va" \ar@{-} ;
 "d" ;"vb" \ar@{-} \ar@{};
 "va";"v2" \ar@{=} ;"va";"v2" \ar@{=} ;
 "vb";"v2" \ar@{=} ;
 "v2" ;"v3" \ar@{=};
 "vc" ;"v3" \ar@{=} ;
 "v3" ;"v4" \ar@{=} \ar@{};
 %
 %
 "va";"v2" \ar@{-} ;"va";"v2" \ar@{-} ;
 "vb";"v3" \ar@{-} ;
 "vc" ;"v3" \ar@{-} ;
 "v3" ;"v4" \ar@{-} ;
 \endxy
$$
\end{small}

Now $M^{w_0, \dots, w_{3}}_0(t) = (0, (w_0+w_1)(w_2+w_3), w_2w_3). $
 To motivate this new weighted version of our algorithm we mention that the weights
 $w_0,\dots,w_{n-1}$ are to be thought of as the sizes of various trees to be grafted to the
 respective leaves. This weighting is therefore necessary to make the induction go through, since
 the induction is itself based upon the grafting of trees.

 Since we are proving that the points $M^{w_0, \dots, w_{n-1}}_0(t)$ are the vertices of the
composihedron, we need to define hyperplanes  $H^{w_0, \dots, w_{n-1}}_0(t)$ for this weighted
version which we will show to be the the bounding hyperplanes when $t$ is a facet tree.
\begin{definition}
Recall that the lower facets $\cal{CK}_k(n-1,2)$ correspond to minimal lower trees.
These are assigned a hyperplane $H^{w_0, \dots, w_{n-1}}_0(l(k,2))$ determined by the equation
$$x_k  =0 .$$
\end{definition}
\begin{lemma}\label{low}
For any painted binary tree $t$ the point $M^{w_0, \dots, w_{n-1}}_0(t)$ lies in the hyperplane
$H^{w_0, \dots, w_{n-1}}_0(l(k,2))$ iff $t$ is domain equivalent to a refinement of $l(k,2).$ Also the hyperplane
$H^{w_0, \dots, w_{n-1}}_0(l(k,2))$ bounds below the points $M^{w_0, \dots, w_{n-1}}_0(t)$ for $t$
any binary representative of a domain class of painted trees.
\end{lemma}
\begin{proof}
The $k^{th}$ node of any binary tree that is a refinement of $l(k,2)$ will be unpainted, and will have left
and right subtrees with only one leaf each. Thus the $k^{th}$ node of
a binary tree that is domain equivalent to
a refinement of $l(k,2)$ will be unpainted. Furthermore, any binary tree with its $k^{th}$ node unpainted
is domain equivalent to
a refinement of $l(k,2),$ by a series of domain equivalence moves.
(Recall that no branches may be painted above an unpainted node.) Thus
any binary tree $t$ which is domain equivalent to a refinement of the lower tree
$l(k,2)$ will yield a point $M^{w_0, \dots, w_{n-1}}_0(t)$ for which $x_k=0$
since the $k^{th}$ node will be unpainted in any such tree. Also,
 if a binary tree $t$ is not equivalent to a refinement of a lower tree $l(k,2)$ then the
point $M^{w_0, \dots, w_{n-1}}_0(t)$ will have the property that
$x_k  > 0,$ since the $k^{th}$ node will be painted.
\end{proof}
\begin{definition}
\end{definition}
Recall that the upper facets $\cal{CK}(t;r_1,\dots ,r_t)$ correspond to upper trees such as:
 \begin{small}
 $$
u(t;r_1,\dots, r_t) = \xy  0;/r.45pc/:
  (-.5,23)*{{}^0}="a"; (1.5,25.25) *{{r_1 \atop \overbrace{~~~}}}; (1.5,23) *{\dots}; (3.5,23)*=0{}="b";
(4.5,23)*=0{}="a2"; (6.5,25.25) *{{r_2 \atop \overbrace{~~~}}}; (6.5,23)  *{\dots};
(8.5,23)*=0{}="b2";
 (12,23)*=0{}="a3"; (14,25.25) *{{r_t \atop \overbrace{~~~}}}; (14,23) *{\dots~~}; (16,23)*{~~^{~n-1}}="b3";
   (2,20)*=0{\bullet}="d";
  (6,20)*=0{\bullet}="e"; (14,20)*=0{\bullet}="f";
 (9,17) *{\dots};
  (8,14)*=0{\bullet}="v2";
  (8,8)*=0{}="v4"; \ar@{};
 "d" ;"v2" \ar@{-}; "d" ;"v2" \ar@{-};
"a" ;"d" \ar@{-};
 "b" ;"d" \ar@{-};
  "a2" ;"e" \ar@{-};
   "b2" ;"e" \ar@{-};
 "a3" ;"f" \ar@{-};
 "b3" ;"f" \ar@{-};
 "e" ;"v2" \ar@{-};
 "f" ;"v2" \ar@{-};
  "v2" ;"v4" \ar@{-};
 \ar@{} ;
"d" ;"v2" \ar@{=};"d" ;"v2" \ar@{=};
 "e" ;"v2" \ar@{=};
 "f" ;"v2" \ar@{=};
 "v2" ;"v4" \ar@{=};
 \endxy
$$
\end{small}

These are assigned a hyperplane $H^{w_0, \dots, w_{n-1}}_0(u(t;r_1,\dots, r_t))$   determined by
the equation
 $$x_{r_1} + x_{(r_1+r_2)} + \dots + x_{(r_1+r_2+\dots +r_{t-1})} = \sum_{1\le i<j \le t}R_iR_j.$$
where $R_i = \sum_{}w_j$ where the sum is over the leaves of the $i^{th}$ subtree  (from left to
right) with root the type (5) node; the index $j$ goes from $(r_1+r_2+\dots +r_{i-1})$ to
$(r_1+r_2+\dots +r_{i}-1)$ (where $r_0 = 0.$)
  Note that if $t=n$ (so $r_i = 1$ for all $i$) that
this becomes the hyperplane given by
$$x_1 + \dots +
x_{n-1} = \sum_{1\le i<j \le n-1}w_iw_j.$$
\newline
\begin{lemma}\label{up}
For any painted binary tree $t$ the point $M^{w_0, \dots, w_{n-1}}_0(t)$ lies in the hyperplane
$H^{w_0, \dots, w_{n-1}}_0(u(t;r_1,\dots, r_t))$ iff $t$ is
(domain equivalent to)
a refinement of $u(t;r_1,\dots, r_t).$
Also the hyperplane $H^{w_0, \dots, w_{n-1}}_0(u(t;r_1,\dots, r_t))$ bounds above the points
$M^{w_0, \dots, w_{n-1}}_0(t)$ for $t$ any binary painted tree.
\end{lemma}
\begin{proof}
These hyperplanes are precisely those defined as $H^{w_0, \dots, w_{n-1}}_q(u(t;r_1,\dots, r_t))$
in Definition 5.7 of \cite{multi}.
   Thus the proof of Lemma 5.8 of \cite{multi}
holds here as well. Letting $q=0$ in the proof of that
 lemma does not change the argument, since in neither definition does the value of $q$ actually appear.
 In fact the only use of $q$ in the proof relies on the fact that $q<1,$ and when $q=0$ this is certainly also
true.
 \end{proof}

\begin{proof} of Theorem~\ref{main}:
Now we may proceed with our inductive argument. The base case of $n = 2$ leaves is trivial to
check. The points in {\bf R}$^{1}$ are $w_0w_1$ and $0$ Their convex hull is a line segment,
combinatorially equivalent to $\cal{CK}(2).$ Now we assume that for all $i<n$
 and for positive integer weights $w_0,\dots, w_{i-1},$ that the convex hull of the
points $\{M^{w_0, \dots, w_{i-1}}_0(t)~|~ t \text{ is a painted binary tree with $i$ leaves}\}$ in
{\bf R}$^{i-1}$ is combinatorially equivalent to the complex $\cal{CK}(i),$ and that the points
$M^{w_0, \dots, w_{i-1}}_0(t)$ are the vertices of the convex hull. Now for $i=n$ we need to show
that the equivalence still holds. Recall that the two items we plan to demonstrate are that the
points $M^{w_0, \dots, w_{n-1}}_0(t)$ are the vertices of their convex hull and that the facet of
the convex hull corresponding to a given lower or upper tree $T$ is the convex hull of just the
points corresponding to the binary trees that are refinements of $T.$ The first item will be seen
in the process of checking the second.

Given an $n$-leaved lower tree $l(k,2)$ we have from Lemma~\ref{low} that the points corresponding
to binary refinements of $l(k,2)$ lie in an $n-2$ dimensional hyperplane $H^{w_0, \dots,
w_{n-1}}_0(l(k,2))$ which bounds the entire convex hull. To see that this hyperplane does indeed
contain a facet of the entire convex hull we use the induction hypothesis to show that the
dimension of the convex hull of just the points in $H^{w_0, \dots, w_{n-1}}_0(l(k,2))$ is $n-2.$
For $q \in (0,1)$ we know  from  \cite{multi} (proof of Theorem 3.1) that the points
 $M^{w_0, \dots, w_{n-1}}_q(t)$ corresponding to the
binary trees $t < l(k,2)$ are the
vertices of a polytope combinatorially equivalent to ${\cal J}(n-1) \times {\cal K}(2) = {\cal J}(n-1).$
Taking the limit as $q\to 0$ piecewise allows us to use the induction hypothesis to show that
the points
 $M^{w_0, \dots, w_{n-1}}_0(t)$ corresponding to the
binary trees $t < l(k,2)$ are the
vertices of a polytope combinatorially equivalent to $\cal{CK}(n-1),$ which is the required type of polytope.

Given an $n$-leaved upper tree $u(t,r_1,\dots,r_t)$ we have from Lemma~\ref{up} that the points
corresponding to binary refinements of $u(t,r_1,\dots,r_t)$ lie in an $n-2$ dimensional hyperplane
$H^{w_0, \dots, w_{n-1}}_0(u(t,r_1,\dots,r_t))$ which bounds the entire convex hull. To see that
this hyperplane does indeed contain a facet of the entire convex hull we use the induction
hypothesis to show that the dimension of the convex hull of just the points in
 $H^{w_0,\dots,w_{n-1}}_0(u(t,r_1,\dots,r_t))$ is $n-2.$

 For $q \in (0,1)$ we know  from  \cite{multi} (proof of Theorem 3.1) that the points
 $M^{w_0, \dots, w_{n-1}}_q(t)$ corresponding to the
binary trees $t < u(t,r_1,\dots,r_t)$ are the
vertices of a polytope combinatorially equivalent to
${\cal K}(t) \times {\cal J}(r_1) \times \dots \times {\cal J}(r_t).$
Here the vertices of the associahedron do not depend on the factor $q.$
Taking the limit as $q\to 0$ piecewise allows us to use the induction hypothesis to show that
the points
 $M^{w_0, \dots, w_{n-1}}_0(t)$ corresponding to the
binary trees $t < u(t,r_1,\dots,r_t)$ are the
vertices of a polytope combinatorially equivalent to
${\cal K}(t) \times \cal{CK}(r_1) \times \dots \times \cal{CK}(r_t).$ This is the required type of polytope.

Since  each $n$-leaved binary painted tree is a refinement of some upper and or or lower trees,
then the point associated to that tree is found as a vertex of some of the facets of the entire
convex hull, and thus is a vertex of the convex hull. This completes the proof.

\end{proof}



\end{document}